\documentclass[11pt]{article}
\usepackage{amsmath, amssymb, amsthm}
\usepackage{graphicx} 
\usepackage{enumitem}                                 
\usepackage[margin=1.0in]{geometry}

\AtEndDocument{\bigskip{\footnotesize%
\textsc{Department of Mathematics, the Pennsylvania State University, University
Park, PA 16802, USA} \par 
  \textit{E-mail address} \texttt{ aka5983@psu.edu
} \par
\textit{E-mail address} \texttt{ hertz@math.psu.edu
} \par 
\textit{E-mail address} \texttt{ zhirenw@psu.edu
}
}}

\newtheorem{theorem}{Theorem}[section]

\newtheorem{Counter-example}[theorem]{Counter example}
\newtheorem{Claim}[theorem]{Claim}

\newtheorem{Lemma}[theorem]{Lemma}
\newtheorem{Proposition}[theorem]{Proposition}

\newtheorem{Definition}[theorem]{Definition}
\newtheorem{Corollary}[theorem]{Corollary}

\newtheorem*{theorem*}{Theorem}

\newcommand{\supp}{\text{supp}}
\newcommand{\diam}{\text{diam}}

\newcommand{\ttau}{{\tilde{\tau}}}

\newcommand{\cA}{{\mathcal A}}

\newcommand{\tA}{{\widetilde A}}

\title{Logarithmic Fourier decay for  self conformal measures}
\author{Amir Algom, Federico Rodriguez Hertz, and Zhiren Wang}
\date{}
\begin{document}
\maketitle
\begin{abstract}
We prove that the Fourier transform of a self conformal measure on $\mathbb{R}$   decays to $0$ at infinity at a logarithmic rate, unless the following  holds: The underlying IFS is smoothly conjugated to an IFS that both acts linearly on its attractor and contracts by scales that are not Diophantine.  Our key technical result is an effective version of a local limit Theorem for cocycles with moderate deviations due to Benoist-Quint (2016), that is of independent interest. 
 \end{abstract}

\section{Introduction}

Let $\nu$ be a Borel probability measure on $\mathbb{R}$. For every $q\in \mathbb{R}$   the Fourier transform of $\nu$ at $q$ is defined by 
\begin{equation*} 
\mathcal{F}_q (\nu) := \int \exp( 2\pi i q x) d\nu(x).
\end{equation*} 
The measure $\nu$ is called a \textit{Rajchman measure} if $\lim_{|q|\rightarrow \infty} \mathcal{F}_q(\nu)=0$. It is a consequence of the Riemann-Lebesgue Lemma that if $\nu$ is absolutely continuous then it is Rajchman. On the other hand, by  Wiener's Lemma  if $\nu$ has an atom then it is not  Rajchman. For measures that are both continuous (no atoms) and singular, determining whether or not $\nu$ is a Rajchman measure may be a challenging problem even for well structured measures. The Rajchman property has various geometric consequences on the measure $\nu$ and its support, e.g. regarding the uniqueness problem    \cite{li2019trigonometric}. Further information about the rate of decay of $\mathcal{F}_q(\nu)$ has even stronger geometric consequences. For example,  by a classical Theorem of Davenport-Erd\H{o}s-LeVeque \cite{Davenport1964Erdos},  establishing a sufficiently fast rate of decay for $F_q (\nu)$ is one means towards finding normal numbers in the support of $\nu$. For some further applications of the Rajchamn property and the rate of decay, see the survey \cite{Lyons1995survey}.

The goal of this paper is to prove  that a wide class of fractal measures  enjoy logarithmic Fourier decay, assuming some mild conditions are met:  Let $\Phi= \lbrace f_1,...,f_n \rbrace$ be a finite set of strict contractions of a compact interval $I\subseteq \mathbb{R}$ (an \textit{IFS} - Iterated Function System), such that every $f_i$ is differentiable.  We say that  $\Phi$ is $C^\alpha$ smooth if every $f_i$ is at least $C^\alpha$ smooth for some $\alpha\geq 1$.  It is well known that there exists a unique compact set $\emptyset \neq K=K_\Phi \subseteq I$ such that
\begin{equation} \label{Eq union}
K = \bigcup_{i=1} ^n f_i (K).
\end{equation}
The set $K$ is called  the \textit{attractor} of the IFS $\lbrace f_1,...,f_n \rbrace$.  We always assume that there exist $i,j$ such that $x_i \neq x_j$, where $x_i$ is the fixed point of $f_i$. This ensures that $K$ is infinite. We call $\Phi$  \textit{uniformly contracting} if 
$$0< \inf \lbrace |f '(x)|:\, f\in \Phi, x\in I \rbrace \leq \sup \lbrace |f '(x)| :\, f\in \Phi, x\in I  \rbrace <1.$$
Next, writing $\mathcal{A}= \lbrace 1,...,n\rbrace$, for every $\omega \in \mathcal{A} ^\mathbb{N}$ and $m\in \mathbb{N}$ let
$$f_{\omega|_m} := f_{\omega_1} \circ \circ \circ f_{\omega_m}.$$
Fix $x_0 \in I$.  Then we have a surjective coding map $\pi: \mathcal{A} ^\mathbb{N} \rightarrow K$ defined by
$$\omega \in \mathcal{A}^{\mathbb N} \mapsto x_\omega:= \lim_{m\rightarrow \infty}  f_{\omega|_m}  (x_0),$$
which is a well defined map because of uniform contraction (see e.g. \cite[Section 2.1]{bishop2013fractal}). 

Let $\textbf{p}=(p_1,...,p_n)$ be a strictly positive probability vector, that is, $p_i >0$ for all $i$ and $\sum_i p_i =1$, and let $\mathbb{P}=\mathbf{p}^\mathbb{N}$ be the corresponding Bernoulli measure on $\mathcal{A}^\mathbb{N}$. We call the measure $\nu = \pi \mathbb{P}$ on $K$ the self conformal measure corresponding to $\mathbf{p}$, and note that our assumptions are known to imply that it is non-atomic.  Equivalently, $\nu$ is the unique Borel probability  measure on $K$ such that
$$\nu = \sum_{i=1} ^n p_i\cdot  f_i\nu,\quad \text{ where } f_i \nu \text{ is the push-forward of } \nu \text{ via } f_i.$$ 
When all the maps in $\Phi$ are affine we call $\Phi$ a self-similar IFS and $\nu$ a self-similar measure.

Next, we say that a $C^1$ IFS $\Phi$ is \textit{Diophantine} if there are $l , C>0$ such that  
\begin{equation} \label{Eq new Dio condition}
\inf_{y\in \mathbb{R}} \max_{i\in \lbrace 1,...n\rbrace } d\left(\,\log |f'_i (x_i)|\cdot x+y,\, \mathbb{Z} \right) \geq \frac{C}{|x|^{l}},  \text{ for all } x\in \mathbb{R} \text{ large enough in absolute value.}
\end{equation}
This condition is adopted from the work of Breuillard \cite{Breuillard2005llt} on effective local limit Theorems for classical random walks on $\mathbb{R}^d$, and serves a similar purpose for us as well. Note that it is invariant under conjugation by $C^1$ maps with non-vanishing derivative.   Next, we say that a $C^2$ IFS $\Psi$ is \textit{linear} if $g''(x) = 0$  for every  $x\in K_\Psi$  and $g\in \Psi$. Note that if $\Psi$ is $C^\omega$ and linear  then  it must be self-similar. We believe it is possible to construct a linear $C^2$ IFS that includes maps with non-locally constant derivative on the attractor, and we hope to discuss this in a future work.

Let $\mathcal{L}$ denote the family of all Borel probability measures on $\mathbb{R}$ that have logarithmic Fourier decay. That is, writing $\mathcal{P}(\mathbb{R})$ for the family of Borel probability measures on $\mathbb{R}$,
$$\mathcal{L} := \lbrace  \mu:\, \mu \in \mathcal{P}(\mathbb{R}) \text{ and there exists }  \alpha>0 \text{ such that } \left|\mathcal{F}_q (\nu)\right| \leq O\left( \frac{1}{ \left| \log |q| \right| ^\alpha} \right), \text{ as } |q|\rightarrow \infty \rbrace.$$
The following Theorem is the main result of this paper. We say that an IFS $\Phi$ is $C^r$ conjugate to an IFS $\Psi$ if there is a $C^r$ diffeomorphism $h$ such that  $\Phi = \lbrace h\circ g \circ h^{-1}\rbrace_{g\in \Psi}$. 
\begin{theorem} \label{Main Theorem}
Let $\Phi$ be an orientation preserving uniformly contracting $C^{r}$   IFS, where $r\geq 2$. If there exists a self conformal measure that is not in $\mathcal{L}$ then $\Phi$ is $C^r$ conjugate to a linear  non-Diophantine IFS.
\end{theorem}
Several remarks are in order: First,  Theorem \ref{Main Theorem} improves our previous work \cite[Theorem 1.1 and Corollary 1.2]{algom2020decay} by establishing a rate of decay in many new cases (our previous work was effective only  for Diophantine self similar IFSs). Secondly, the orientation preserving assumption is made purely for notational convenience, and can be easily dropped.   Finally, we emphasize that no separation conditions are imposed on $\Phi$.

Recent years have seen an explosion of interest and progress regarding the study of Fourier decay for fractal measures. We proceed to give a concise overview of results related to Theorem \ref{Main Theorem}, and refer to \cite[Section 1]{algom2020decay} for more details on e.g. the methods involved: Combining the work of Bourgain-Dyatlov \cite{Bour2017dya} with \cite{Li2018decay}, Li \cite{li2018fourier} proved the Rajchman property for Furstenberg measures for $SL(2,\mathbb{R})$ cocycles under mild assumptions (there are known conditions that ensure that such measures are self-conformal \cite{Yoccoz2004some, Avila2010jairo}).   Sahlsten-Stevens \cite{sahlsten2020fourier, sahlsten2018fourier} proved the Rajchman property  for Gibbs measures on $C^\omega$  self-conformal sets under some additional assumptions. These include the strong separation condition (i.e. that the union \eqref{Eq union} is disjoint), and a stronger non-linearity assumption: The IFS is not conjugate to an IFS where the derivatives of the maps are locally constant on its attractor. Our previous work \cite[Corollary 1.2 part (3)]{algom2020decay} gave a unified proof of the Rajchman property for many of these cases, and   Theorem \ref{Main Theorem} further upgrades this result by establishing a logarithmic rate of decay. On the other hand,  Bourgain-Dyatlov, Li, and Sahlsten-Stevens, establish  a polynomial rate of decay, but these works require various further assumptions. We believe that  when $\Phi$ is not $C^r$ conjugate to a linear IFS then the assumptions of Theorem \ref{Main Theorem} should ensure that all self conformal measures  have polynomial Fourier decay. See the end of this introduction for some more discussion about this issue.

Next,  suppose $\Phi$ is a $C^{2}$ IFS that is smoothly conjugated to a self similar IFS with contractions ratios $\lbrace r_1,...,r_n \rbrace \subset \mathbb{R}_+$ such that: There exist $C>0,l>2$ with
\begin{equation} \label{Eq Dio condition 2}
 \max_{i\in \lbrace 1,...n\rbrace } d(\,\log |r_i|\cdot x,\, \mathbb{Z}) \geq \frac{C}{|x|^{l}}, \text{ for all } x\in \mathbb{R} \text{ large enough in absolute value.}
\end{equation}
Then, by \cite[Remark 6.7]{algom2020decay}, there exists a $C^{2}$ IFS $\Psi$ as in Theorem \ref{Main Theorem} that satisfies \eqref{Eq new Dio condition}, and every self conformal measure with respect to $\Phi$ is also self-conformal with respect to $\Psi$.  So, in the conjugate-to-self-similar situation, it is enough to assume the self-similar IFS meets  condition \eqref{Eq Dio condition 2}  in order for all self-conformal measures to  be in $\mathcal{L}$. This generalizes an effective decay result of Li-Sahlsten  \cite[Theorem 1.3]{li2019trigonometric} for self-similar measures.

In the context of self-similar IFSs,   when all contraction ratios are powers of some $r\in (0,1)$,  Varj\'{u}-Yu  \cite{varju2020fourier} proved logarithmic decay as long as $r^{-1}$ is not a Pisot or a Salem number. Kaufman \cite{Kaufman1984ber} and Mosquera-Shmerkin \cite{Shmerkin2018mos} proved polynomial Fourier decay for $C^2$ non-linear IFS's that arise by conjugating homogeneous (that are never Diophantine) self-similar IFS's.   Solomyak \cite{Solomyak2021ssdecay, solomyak2021fourier} has recently shown that in fact, outside a zero Hausdorff dimension exceptional set of parameters,   self-similar measures on $\mathbb{R}$ and certain self-affine measures always have polynomial Fourier decay. Br\'{e}mont \cite{bremont2019rajchman} recently resolved  the Rajchman problem for self-similar measures on $\mathbb{R}$, and  Rapaport \cite{rapaport2021rajchman} extended this to self-similar measures on $\mathbb{R}^d$ for any $d\geq 1$ (see also \cite{Li2020Sahl}).  Finally, we mention the classical  work of Erd\H{o}s \cite{Erdos1940ber} and Kahane \cite{Kahane1971Ber} about polynomial decay being typical for Bernoulli convolutions, and the more recent  works \cite{Dai2012ber, Buf2014Sol, Dai2007Feng} about rates in some explicit examples of Bernoulli convolutions.

Let us now outline the proof of Theorem \ref{Main Theorem}, and along the way describe the orgnization of this paper. Fix $\Phi$ as in Theorem \ref{Main Theorem}, and assume it is either Diophantine or not-conjugate-to-linear. We aim to show that all self-conformal measures are in $\mathcal{L}$, which implies Theorem \ref{Main Theorem} since the Diophantine condition \eqref{Eq new Dio condition} is invariant under smooth conjugation. We begin with Section \ref{Section transfer}, where we define the derivative cocycle of the IFS and the  transfer operator corresponding to it and to a fixed probability vector $\mathbf{p}$ as above, and  recall some known results about it. We then proceed to prove Theorem \ref{Thm Dolgopyat}, an  estimate on the norm of iterations of the  transfer operator, which requires some delicate analysis that is closely related to the work of Dolgopyat \cite{Dolgopyat1998rapid}. In particular, in the not-conjugate-to-linear case  we will make use of the so-called temporal distance function \cite[Appendix A.1]{Dolgopyat1998rapid}. This is, to the best of our information, the first such anylsis to be done in the context of general $C^2$ IFS's without separation (in the presence of separation there are numreous papers that conduct similar analyses, e.g. the work of Naud \cite{Naud2005exp} for separated $C^\omega$ IFS's).

Afterwards, in Section \ref{Section LLT and CLT}, we show that  certain random walks driven by the derivative cocycle  satisfy an effective version of the central limit Theorem. This is Theorem \ref{Theorem CLT kasun}, that follows from a standard application of the Nagaev-Guivarc'h method as presented in the work of Gou\"{e}zel \cite{Gouzel2015limit}. Thus, all we have to do to this end is to verify that the conditions of  \cite[Theorem 3.7]{Gouzel2015limit} are met, which is a consequence of well known results that are discussed in Section \ref{Section prop of trans}.

Section \ref{Section LLT} then contain  the most subtle step towards Theorem \ref{Main Theorem}, and the main technical result of this paper: We prove an \textit{effective} version of Benoist-Quint's local limit Theorem with moderate deviations \cite[Theorem 16.1]{Benoist2016Quint} for random walks driven by the derivative cocycle. Here we combine our estimates on the contraction properties of the transfer operator obtained in Theorem \ref{Thm Dolgopyat} with the work of Breuillard \cite{Breuillard2005llt}, who proved effective local limit Theorems for classical random walks on $\mathbb{R}^d$ under a Diophantine condition similar to \eqref{Eq new Dio condition}, and with the work of Benoist-Quint \cite[Chapter 16]{Benoist2016Quint}, to derive our local limit Theorem \ref{Theorem LLT}.

In Section \ref{Section Thm equid} we  use these effective limit Theorems to obtain a certain effective conditional local limit Theorem for the derivative cocycle. This is Theorem \ref{Theorem equid}, which is an upgraded version of our previous result \cite[Theorem 3.7]{algom2020decay} as it is effective (holds with a polynomial rate). Finally, in Section \ref{Section proof}, we show that all self conformal measures  belong to $\mathcal{L}$. To this end we combine Theorem \ref{Theorem equid} with a delicate linerization scheme, and a more robust estimation of certain oscillatory integrals as in \cite[Section 4.2]{algom2020decay}.

Finally, we remark that in the not-conjugate-to-linear case it might be possible to further upgrade our local limit Theorem \ref{Theorem LLT} to hold with an exponential rate of convergence. This would be an important step towards showing that in this case all self-conformal measures have polynomial Fourier decay. Also,  it is  possible that Theorem \ref{Main Theorem} is optimal in the Diophantine case, since such IFSs may be self-similar, where much less is known regarding polynomial Fourier decay in concrete cases (though "most" self-similar measures do have polynomial decay as shown in the afformentioned work of Solomyak \cite{Solomyak2021ssdecay}).

\noindent{ \textbf{Acknowledgements}} We thank Tuomas Sahlsten and Kasun Fernando for some helpful  discussions. We  are also grateful to the anonymous referee for a very thorough reading and
many helpful suggestions which greatly improved the presentation of the paper

\section{The derivative cocycle and the associated transfer operator} \label{Section transfer}
\subsection{Preliminaries} \label{Section pre transfer}
Fix an orientation preserving $C^{2}$ IFS $\Phi = \lbrace f_1,...,f_n \rbrace$ and write $\mathcal{A}= \lbrace 1,...,n\rbrace$.  For every $1\leq a\leq n$ let $\iota_a: \mathcal{A}^{\mathbb N}\to \mathcal{A} ^{\mathbb N}$ be the map 
$$\iota_a(\omega_1,\omega_2,\cdots)=(a,\omega_1,\omega_2,\cdots).$$
Let  $G$ to be the free semigroup generated by the family $\{\iota_a: 1\leq a\leq n\}$, which acts on $\mathcal{A} ^{\mathbb N}$ by composing the corresponding $\iota_a$'s. We define the derivative cocycle $c:G\times \mathcal{A} ^{\mathbb N}\rightarrow \mathbb{R}$ via
\begin{equation} \label{The der cocycle}
c(a,\omega)=-\log f'_a(x_\omega).
\end{equation}
Let $\rho:= \sup_{f\in \Phi} ||f'||_\infty \in (0,1)$, and  define a metric on $\mathcal{A}^{\mathbb N}$ via 
\begin{equation} \label{Eq d rho}
d_\rho (\omega,\omega'):=\rho^{\min\{n:\ \omega_n\neq\omega'_n\}}.
\end{equation}
We record the following standard Claim for future use:
\begin{Claim} \label{Properties of cocycle}
For every $a\in \mathcal{A}$ the following statements hold true:
\begin{enumerate}
\item The map $\iota_a$ is uniformly contracting: 
$$d(\iota_a(\omega),\iota_a(\eta))=\rho d(\omega,\eta).$$

\item The cocycle $c(a,\omega)$ is  uniformly bounded, Lipschitz in $\omega$, with a uniformly bounded Lipschitz constant as $a \in \mathcal{A}$ varies.
\end{enumerate}
\end{Claim} 
This is standard, since all the maps in $\Phi$ are $C^{2}$ smooth, and since by uniform contraction
\begin{equation} \label{Eq C and C prime}
0<D:= \min \lbrace -\log |f' (x)| : f\in \Phi, x\in I \rbrace, \quad D':= \max \lbrace -\log |f' (x)| : f\in \Phi, x\in I \rbrace < \infty.
\end{equation}

Next, let $H^1 = H^1 (\rho)$ denote the space of Lipschitz functions  $\mathcal{A}^\mathbb{N} \rightarrow \mathbb{C}$ in the metric $d_\rho$,  and equip $H^1$ with the norm
\begin{equation} \label{Eq B-Q norm}
|\varphi|_1 = ||\varphi||_\infty + c_1 (\varphi),\text{ where } c_1(\varphi)= \sup_{\omega \neq \omega'} \frac{|\varphi(\omega)-\varphi(\omega')|}{d_\rho (\omega,\omega')} = \text{ the Lipschitz constant of } \varphi.
\end{equation}
Following Dolgopyat \cite[Section 6]{Dolgopyat1998rapid}, for every $\theta\neq 0$ we  define yet another norm on $H^1$ via
\begin{equation} \label{Eq norm dol theta}
||\varphi||_{(\theta)} = \max\lbrace ||\varphi||_\infty, \, \frac{c_1 (\varphi)}{2 C_6 |\theta|}  \rbrace
\end{equation}
for a  constant $C_6>0$ whose exact choice will be explained soon.

 Next, let $\textbf{p}=(p_1,...,p_n)$ be a strictly positive probability vector on $\mathcal{A}$, and let $\mathbb{P}=\mathbf{p}^\mathbb{N}$ be the corresponding product measure on $\mathcal{A}^\mathbb{N}$. Note that $\mathbb{P}$ is the unique stationary measure corresponding to the measure $\mu := \sum_{a\in \mathcal{A}} p_a \cdot \delta_{\lbrace \iota_a \rbrace}$ on $G$.
 \begin{Definition} \label{Def transfer operator}
 For every $\theta \in \mathbb{R}$ let $P_{i\theta}:H^1 \rightarrow H^1$ denote the transfer operator defined by, for $\phi \in H^1$ and $\omega \in \mathcal{A}^\mathbb{N}$,
$$P_{i\theta} (\phi)(\omega) = \int e^{2\pi i \theta c(a,\omega)} \phi(\iota_a(\omega)) d\mathbf{p}(a). $$
 \end{Definition}
We can now remark that the constant $C_6>0$ is chosen so that $|| P_{i\theta} ^n ||_{(\theta)} \leq 1$ for all $n$  - see \cite[Proposition 2]{Dolgopyat1998rapid} for more details.

\subsection{Some properties of the transfer operator}  \label{Section prop of trans}
In this Section we recall some properties of the family of operators $\lbrace P_{i\theta} \rbrace_{\theta \in \mathbb{R}}$, working with the norm \eqref{Eq B-Q norm} on $H^1$. We begin with following standard results:
\begin{theorem}  \label{Theorem B-Q 1}
Suppose
$\Phi$ satisfies the conditions of Theorem \ref{Main Theorem}. Let $\mathbb{P}=\mathbf{p}^\mathbb{N}$ be a Bernoulli measure on $\mathcal{A}^\mathbb{N}$. Then the following properties hold true:
\begin{enumerate}
\item \cite[Lemma 11.17]{Benoist2016Quint} $P_{i\theta}$ is an analytic function of $\theta$. 

\item  \cite[Lemma 15.1 and Lemma 15.3]{Benoist2016Quint} The constant function $\mathbf{1} \in H^1$  is an isolated and simple eigenvalue of $P_{i0}$.  All other eigenvalues of $P_{i0}$ have absolute value less than $1$, and its essential spectrum is strictly contained inside the unit disc. 
 
\end{enumerate}

\end{theorem}
Let us take a moment to explain how our setup fits into the more general one outlined in the work of  Benoist-Quint \cite{Benoist2016Quint}: With the notations of \cite[Chapter 11]{Benoist2016Quint}, our acting semigroup is $G$ as in the beginning of Section \ref{Section pre transfer}, $F$ is the trivial group (and so is the morphism $s$), and $E$ is simply taken to be $\mathbb{R}$. Recalling the definition of the measure $\mu$ on $G$ from before Definition \ref{Def transfer operator}, the compact metric $G$-space on which $G$ is $(\mu,1)$-contracting is taken to be $\mathcal{A}^\mathbb{N}$ (this follows from Claim \ref{Properties of cocycle} part (1)), and recall that $\mathbb{P}$ is the unique stationary measure. Since $\mu$ is finitely supported, via Claim \ref{Properties of cocycle} part (2) our cocycle $c$ trivially has both finite exponential moment and its Lipschitz constant has finite moment \cite[Eq. (11.14) and (11.15)]{Benoist2016Quint}. Thus, in our setup \cite[Lemmas 11.17, 15.1, 15.3]{Benoist2016Quint} can all be applied. Now, \cite[Lemma 11.17]{Benoist2016Quint} immediately implies Theorem \ref{Theorem B-Q 1} part (1). Since, by   \cite[Equation (15.3)]{Benoist2016Quint} the only eigenfunction of modulus $1$ of $P_{i0}$ is $\mathbf{1}$, \cite[Lemma 15.1 and Lemma 15.3]{Benoist2016Quint} imply Theorem \ref{Theorem B-Q 1} part (2).

We proceed to recall some results proved by Benoist-Quint \cite{Benoist2016Quint} regarding certain contraction properties of $P_{i\theta}$:   For every small enough $\theta$  the operator $N_{i\theta}:H^1 \rightarrow H^1$ is defined in \cite[Lemma 11.18]{Benoist2016Quint} as an analytic continuation of the operator $N_0 (\varphi) = N(\varphi)= \mathbb{P} ( \varphi)$. Furthermore, $N_{i\theta}$ is the projection onto the one dimensional eigenspace spanned by the eigenvector with the leading eigenvalue $\lambda_{i\theta}$ of $P_{i\theta}$.  The local behaviour of $\lambda_{i\theta}$ near $0$ plays a crucial role in the analysis of Benoist-Quint \cite[$\text{Parts } \romannumeral 3 \text{ and }  \romannumeral 4$]{Benoist2016Quint}, and also in our work. 

In the following Proposition we  use the standard re-centring  trick \cite[Equation (3.9)]{Benoist2016Quint} and assume 
\begin{equation} \label{Eq rece cocycle}
 \chi=\chi_\mathbf{p} = \int c(a,\omega) d\mathbf{p}(a) d\mathbb{P} (\omega) =0.
\end{equation} 
Notice that this amounts to changing the cocycle $c$ to a re-centred version $c-\chi$, which only adds a constant phase to $P_{i\theta}$, so it does not affect its norm (note that  $\chi$ equals $\sigma_{\mathbf{p}}$ in the notations of \cite[Equation (3.9)]{Benoist2016Quint}).

\begin{Proposition} (Benoist-Quint) \label{Prop compact and small guys} Assume the conditions of Theorem \ref{Theorem B-Q 1} hold, and  suppose in addition that $\Phi$ is either Diophantine or not-conjugate-to-linear. Then   we have:
\begin{enumerate}
\item \cite[Corollary 15.2]{Benoist2016Quint}  Let $J\subset \mathbb{R}_+$ be a compact set such that $0\notin J$. Then there are $n_0 \in \mathbb{N}$ and $C'\in (0,1)$ such that
$$\sup_{\theta \in J} ||P_{i\theta} ^{n_0} ||_1 < C' <1.$$

\item \cite[Lemmas 11.18 and 11.19]{Benoist2016Quint} For every $\epsilon>0$ small enough there is some constant $C'' \in (0,1)$ such that
$$\sup_{|\theta| \in [0,\epsilon]} ||P_{i\theta} ^{n} ||_1 \leq 2 e^{-C'' \cdot \theta^2 \cdot  n}.$$
 \end{enumerate}
\end{Proposition}
Here we are using the norm from \eqref{Eq B-Q norm} for the operator norm, as in \cite[Chapter 11.3]{Benoist2016Quint}. Now,  part (1) follows from \cite[Corollary 15.2]{Benoist2016Quint} since our assumptions are known to imply that $P_{i\theta}$ does not have an eigenfunction of modulus $1$ for $\theta \neq 0$: This follows from e.g. \cite[Section 6.1]{algom2020decay} in the Diophantine case, and from \cite[Section 6.4]{algom2020decay} in the not-conjugate-to-linear case.  To derive part (2) from \cite[Lemmas 11.18 and 11.19]{Benoist2016Quint}  we need to explain why here the variance $r_0 = r_0 (\mathbf{p})$ of the associated Gaussian  as in the central limit Theorem \cite[Theorem 12.1 part (i)]{Benoist2016Quint} (see also Section \ref{Section LLT and CLT}) satisfies that $r_0>0$: Recall that $I$ is an interval such that every $f\in \Phi$ is a self map of $I$. We can define a derivative cocycle $c'$ directly on $\mathcal{A} \times I$ via
$$c'(i,x)= -\log f_i '(x).$$ 
It is well known that having $r_0 = 0$ implies that the cocycle $c'$ is $C^1$ co-homologous to a constant (see e.g. \cite[Section 6.4]{algom2020decay} for a very similar argument). This  is clearly impossible if $\Phi$ is Diophantine. In addition, if $c'$ is $C^1$ co-homologous to a constant, then a standard argument shows  that $\Phi$ is conjugate to linear (in fact, this argument is included in the proof of Claim \ref{Claim final ing} below). Thus, in our setting $r_0>0$, and so one may use the Taylor-Young formula  for $\lambda_{i\theta}$ obtained via  \cite[Lemmas 11.18 and 11.19]{Benoist2016Quint}  similarly to e.g.  \cite[third paragraph in the proof of Theorem 3.7]{Gouzel2015limit} to derive part (2) (Note: in that proof $\lambda_{i\theta}$ is denoted by $\lambda(t)$). 

\subsection{Contraction properties  of $P_{i\theta}$ for large $\theta$} \label{Section contraction}
As in the work of Dolgopyat \cite{Dolgopyat1998rapid}, for every $\beta>0$ and $\theta \in \mathbb{R}$ let
$$n(\beta, \theta) = [\beta\cdot \log |\theta|].$$
The following Theorem is the key behind the proof of our effective local limit Theorem with moderate deviations, Theorem \ref{Theorem LLT}:

\begin{theorem} \label{Thm Dolgopyat}
Suppose $\Phi$ satisfies the conditions of Theorem \ref{Main Theorem} and is either Diophantine or not-conjugate-to-linear. Then there are $\alpha,\beta,C>0$ such that for every $|\theta|>1$ we have
$$||P_{i\theta} ^{n(\beta,\theta)} ||_{(\theta)} \leq 1- \frac{C}{|\theta|^\alpha}$$
where the operator norm is taken with respect to the norm $||\cdot ||_{(\theta)}$.
\end{theorem}

The proof of Theorem \ref{Thm Dolgopyat} relies  on some ideas going back to the work of  Dolgopyat \cite{Dolgopyat1998rapid}.  First, we will need:
\begin{Lemma} \cite[Lemma 3]{Dolgopyat1998rapid} \label{Lemma 3}
Let $\alpha>0$. If there is some $\beta>0$ such that for every $\theta$ with $|\theta|>1$ and for every $\varphi \in H^1$ with $||\varphi||_{(\theta)} \leq 1$ there exists some $\omega_0 \in \mathcal{A}^\mathbb{N}$ and $0\leq n \leq 3 n(\beta,\theta)$ such that
$$\left|P_{i\theta} ^n \left( \varphi \right)(\omega_0) \right| \leq 1- \frac{1}{|\theta|^\alpha},$$
then there exist $\tilde{\beta}, C_{15}, \alpha_{9}>0$ such that for every $|\theta|>1$
$$||P_{i\theta} ^{n(\tilde{\beta},\theta)} ||_{(\theta)} \leq 1-\frac{C_{15}}{4|\theta|^{\alpha_9}}.$$
\end{Lemma}
We remark that $\alpha_9$ is related to $\alpha$ and to the entropy of  $\mathbb{P}$. Notice that the formal conclusion of \cite[Lemma 3]{Dolgopyat1998rapid} is different from that of Lemma \ref{Lemma 3}. Nonetheless, the conclusion of Lemma \ref{Lemma 3} follows from the proof of \cite[Lemma 3]{Dolgopyat1998rapid}  - which is explicitly stated in the argument (for the readers' convenience we use the same notation $\alpha_9,C_{15}$ as in \cite{Dolgopyat1998rapid}).

Next, we  recall what happens if $\Phi$ fails the conditions of Lemma \ref{Lemma 3}. First, we require the following Definition:
\begin{Definition} \label{Def AAE}
We say $\Phi$ has the approximate eigenfunctions (\textit{AAE}) property if for every $\alpha_0>0$ there are $\alpha,\beta> \alpha_0$ such that one can find arbitrarily large $\theta$ satisfying:

There are  $\Theta = \Theta(\theta) \in \mathbb{R}$ and $H=H_\theta\in H^1$ with $|H(\omega)|=1$ for all $\omega\in \mathcal{A}^\mathbb{N}$, such that:
\begin{equation} \label{Eq AAE}
\left|  e^{ i \theta c\left( \omega|_{n(\beta,\theta)}, \, \sigma^{ n(\beta,\theta)} (\omega) \right)} H\left( \sigma^{n(\beta,\theta)} (\omega) \right)  - e^{i\Theta} H  (\omega)  \right| \leq \frac{1}{|\theta|^\alpha}
\end{equation}
and  the Lipschitz norm of $H$ satisfies
$$\max\lbrace ||H||_\infty, c_1(H) \rbrace \leq O(|\theta|).$$
\end{Definition}
We remark that the terminology \textit{AAE} is adopted from \cite[Section 4.3.2]{Kasun020Pene}. The following Lemma is proved in \cite[Section 8]{Dolgopyat1998rapid}:
\begin{Lemma} \cite[Lemma 4]{Dolgopyat1998rapid} \label{Lemma 4}
If $\alpha>0$ fails the conditions of Lemma \ref{Lemma 3} for every $\beta>0$, then there is some $\beta=\beta(\alpha)>0$  and a sequence $|\theta_k|\rightarrow \infty$ with associated sequences of $\Theta_k \in \mathbb{R}$, $H_k \in H^1$ with $|H_k|=1$, such that \eqref{Eq AAE} holds true for all $k$. Furthermore, $\beta$ can be taken to be arbitrarily large, so if $\Phi$ fails the conditions of Lemma \ref{Lemma 3} for every  $\alpha,\beta>0$ then it has the AAE property.
\end{Lemma}
Notice that  \cite[Lemma 4]{Dolgopyat1998rapid} is stated in terms of approximate eigenfunctions of  iterations of a certain operator defined in \cite[Page 2]{Dolgopyat1998rapid} - our statement avoids this notation, and follows by unwinding Dolgopyat's definitions. Next, a-priori \cite[Lemma 4]{Dolgopyat1998rapid} makes a  different assumption, about the norm of the resolvent operator, but for the proof of \cite[Lemma 5]{Dolgopyat1998rapid} (which is the crucial step in the proof) only \cite[Equation (3)]{Dolgopyat1998rapid} is required - and this is precisely the assumption made in Lemma \ref{Lemma 4}. We remark that Dolgopyat's extra assumption on the norm of the resolvent operator is required for his analysis in \cite[Section 9]{Dolgopyat1998rapid}, which allows him to upgrade the conclusion of Lemma \ref{Lemma 4} into having $\Theta\equiv 0$. We do not know if in our setting such a bound  on the norm of the resolvent operator holds true.

\subsubsection{Proof of Theorem \ref{Thm Dolgopyat} under the Diophantine condition \eqref{Eq new Dio condition}}

We  show that if $\Phi$ satisfies the Diophantine condition \eqref{Eq new Dio condition}  then there is some $\alpha>0$ that  satisfies the conditions of Lemma \ref{Lemma 3}. Thus, via the conclusion of Lemma \ref{Lemma 3}, Theorem \ref{Thm Dolgopyat} will follow. Suppose that $\alpha>0$ fails the conditions of Lemma \ref{Lemma 3} for every $\beta>0$. Then by Lemma \ref{Lemma 4} there is some $\beta=\beta(\alpha)>0$ such that  we can find  a sequence $|\theta_k|\rightarrow \infty$ with associated sequences of $\Theta_k \in \mathbb{R}$, $H_k \in H^1$ with $|H_k|=1$, such that
\begin{equation} \label{Eq sim approx}
\left|  e^{ i \theta_k c\left( \omega|_{n(\beta,\theta_k)}, \, \sigma^{ n(\beta,\theta_k)} (\omega) \right)} H_k\left( \sigma^{n(\beta,\theta)} (\omega) \right)  - e^{i\Theta_k} H  (\omega)  \right| \leq \frac{1}{|\theta_k|^\alpha}.
\end{equation}
Now, for every $a\in \mathcal{A}$ let $\overline{a} \in \mathcal{A}^\mathbb{N}$ be the constant sequence $a$. Let  $x_a$ be the fixed point of $f_a$. It follows from \eqref{Eq sim approx} by plugging in $\omega = \overline{a},a\in \mathcal{A}$, that for every $k$ there is some $y_k \in \mathbb{R}$ (that corresponds to $\Theta_k$),  such that for $m_a\in \mathbb{Z}$ that may differ between the $a$'s,
$$\theta_k\cdot n(\beta,\theta_k)\cdot \log|f'_a (x_a)|+y_k= 2 \pi m_a  +O(|\theta_k|^{-\alpha}).$$ 
Therefore, for all $k$ we get
$$\inf_{y\in \mathbb{R}} \max_{a\in \mathcal{A}} d\left( \frac{1}{2\pi} \cdot \theta_k\cdot n(\beta,\theta_k) \cdot \log|f'_a (x_a)|+y, \, \mathbb{Z} \right)  = O(|\theta_k|^{-\alpha}).$$
On the other hand, by the Diophantine condition  there are  $\ell, C>0$ such that for every $s\in \mathbb{R}$ large enough in absolute value,
$$\inf_{y\in \mathbb{R}} \max_{a\in \mathcal{A}} d\left( s \cdot \log|f'_a (x_a)|+y, \, \mathbb{Z} \right)  \geq \frac{C}{|s|^\ell}.$$

Combining the last two displayed equations and using that as $k\rightarrow \infty$ we have $|\theta_k|\rightarrow \infty$,  we see that $\alpha \leq \ell$. Therefore, for every $\alpha>\ell$ there exists some $\beta>0$ such that the conditions of Lemma \ref{Lemma 3} hold true. This completes the proof of Theorem \ref{Thm Dolgopyat} in this case. \hfill{ $\Box$}

\subsubsection{Proof of Theorem \ref{Thm Dolgopyat} assuming $\Phi$ is not conjugate to linear}
We now prove Theorem \ref{Thm Dolgopyat} assuming $\Phi$ is not conjugate to linear. First, we require the following definition, that is originally due to Chernov \cite{Chernov1998Markov}.
 \begin{Definition} \cite[Appendix A.1]{Dolgopyat1998rapid}
The symbolic  temporal distance function \newline$D:\mathcal{A}^\mathbb{N} \times \mathcal{A}^\mathbb{N} \times \mathcal{A}^\mathbb{N} \times \mathcal{A}^\mathbb{N} \rightarrow \mathbb{R}$  is defined by
$$D(\xi, \zeta,\omega,\eta) :=\lim_n \left( \left( \log f_{\xi|_n} '  (x_\omega) -  \log f_{\xi|_n} '   (x_\eta)  \right) - \left( \log f_{\zeta|_n} '  (x_\omega) -  \log f_{\zeta|_n} '    (x_\eta)  \right) \right). $$
The Euclidean temporal distance function $E:\mathcal{A}^\mathbb{N} \times \mathcal{A}^\mathbb{N} \times I \times I \rightarrow \mathbb{R}$, where $I$ is the interval $\Phi$ is acting on, is defined by
$$D(\xi, \zeta,x,y) :=\lim_n \left( \left( \log f_{\xi|_n} '   (x) -  \log f_{\xi|_n} '   (y)  \right) - \left( \log f_{\zeta|_n} '   (x) -  \log f_{\zeta|_n} '   (y)  \right) \right). $$
\end{Definition}
Notice that $D$ and $E$ are well defined since $\Phi$ is uniformly contracting and $C^2$. The following Theorem is essentially \cite[Theorem 6]{Dolgopyat1998rapid}, with some variations similar to \cite[Theorem 5.6]{Melbourne2018rapid}. For a bounded set $X\subset \mathbb{R}$ we denote its lower box dimension by  $\underline{\dim}_B X$.

\begin{theorem} \label{Theorem AAE and box dim}
If $\Phi$ has the AAE property  then $\underline{\dim}_B D\left( \mathcal{A}^\mathbb{N} \times \mathcal{A}^\mathbb{N} \times \mathcal{A}^\mathbb{N} \times \mathcal{A}^\mathbb{N} \right)=0$.
\end{theorem}
\begin{proof}
First, for every $n\in \mathbb{N}$ and $(\xi, \zeta, \omega,\eta)\in (\mathcal{A}^\mathbb{N})^4$ we define 
$$D_n(\xi, \zeta, \omega,\eta) := \left( \log f_{\xi|_n} '   (x_\omega) -  \log f_{\xi|_n} '   (x_\eta)  \right) - \left( \log f_{\zeta|_n} '   (x_\omega) -  \log f_{\zeta|_n} '  (x_\eta)  \right)$$
and notice that, since $\rho = \sup_{f\in \Phi} ||f'||_\infty < 1$, we have
$$ D(\xi, \zeta,\omega,\eta) = D_n(\xi, \zeta,\omega,\eta) +O\left(   \rho^n \right).$$
Combining this with the definition of $c$,
\begin{eqnarray*}
\exp\left( i \theta D(\xi, \zeta,\omega,\eta) \right) & =& \exp\left( i \theta D_n(\xi, \zeta,\omega,\eta) \right) + O(|\theta| \cdot \rho^n) \\
& =& \frac{ \exp\left( i \theta \log f_{\xi|_n} '  (x_\omega) \right)  }{ \exp \left( i \theta \log f_{\xi|_n} '   (x_\eta) \right) } \cdot \frac{ \exp\left( i \theta \log f_{\zeta|_n} '  (x_\eta) \right)  }{ \exp \left( i \theta \log f_{\zeta|_n} '  (x_\omega) \right) }   + O(|\theta| \cdot \rho^n) \\
&=& \frac{ \exp\left( i \theta c\left( \left(\xi|_n .\omega\right)|_n, \, \sigma^n (\xi|_n .\omega) \right)  \right)  }{ \exp\left( i \theta c\left( \left(\xi|_n .\eta\right)|_n, \, \sigma^n (\xi|_n .\eta) \right)  \right) } \cdot \frac{ \exp\left( i \theta c\left( \left(\zeta|_n .\eta\right)|_n, \, \sigma^n (\zeta|_n .\eta) \right)  \right) }{ \exp\left( i \theta c\left( \left(\zeta|_n .\omega\right)|_n, \, \sigma^n (\zeta|_n .\omega) \right)  \right) } \\
&+&  O(|\theta| \cdot \rho^n). 
\end{eqnarray*}
Let $\alpha_0>0$ be fixed, and let $\alpha,\beta>\alpha_0$. Using the AAE property and the equation above, we can find arbitrarily large $\theta$ and $H=H_\theta\in H^1$ as in Definition \ref{Def AAE} such that we have
$$\exp\left( i \theta D(\xi, \zeta,\omega,\eta) \right)= \frac{H(\xi|_{n(\beta,\theta)}.\omega)}{H(\xi|_{n(\beta,\theta)}.\eta)} \cdot \frac{H(\zeta|_{n(\beta,\theta)}.\eta)}{H(\zeta|_{n(\beta,\theta)}.\omega)} + O(|\theta|^{-\alpha})  +O(|\theta| \cdot \rho^{n(\beta,\theta)}).$$
Since  $n(\beta,\theta)=[\beta \log |\theta|]$, via Lemma \ref{Lemma 4} we may assume $\beta$ is large enough so that we have 
$$|\theta| \rho^{n(\beta,\theta)} \leq \rho^{-1} \cdot |\theta|^{-\alpha}.$$
Therefore, since $|H|\equiv 1$ we have
$$ \left| \frac{H(\xi|_{n(\beta,\theta)}.\omega)}{H(\xi|_{n(\beta,\theta)}.\eta)} - 1 \right| =   \left| H(\xi|_{n(\beta,\theta)}.\omega) - H(\xi|_{n(\beta,\theta)}.\eta) \right| \leq c_1(H)\cdot d_\rho( \xi|_{n(\beta,\theta)}.\eta,\, \xi|_{n(\beta,\theta)}.\omega)$$
$$ \leq O( |\theta| \rho^{{n(\beta,\theta)}}) \leq O(|\theta|^{-\alpha}). $$
Since the same is true for the term corresponding to $\zeta$, it follows that
$$\left| \exp\left( i \theta D(\xi, \zeta, \omega,\eta) \right) - 1 \right| = O(|\theta|^{-\alpha})$$
for arbitrarily large $\theta$ and every $(\xi,\eta,\omega,\eta)$. Thus, there is some $C=C(\alpha)$ such that for arbitrarily large $\theta$,
$$D( \mathcal{A}^\mathbb{N} \times \mathcal{A}^\mathbb{N} \times \mathcal{A}^\mathbb{N} \times \mathcal{A}^\mathbb{N} ) \subseteq \bigcup_{j\in \mathbb{Z}} \left( \frac{2\pi j}{|\theta|}-\frac{C}{|\theta|^{\alpha+1}}, \frac{2\pi j}{|\theta|}+\frac{C}{|\theta|^{\alpha+1}} \right).$$
So, 
$$ \underline{\dim}_B (D( \mathcal{A}^\mathbb{N} \times \mathcal{A}^\mathbb{N} \times \mathcal{A}^\mathbb{N} \times \mathcal{A}^\mathbb{N} )) \leq \frac{1}{\alpha+1}.$$
The result follows since $\alpha$ can be made arbitrarily large.
\end{proof}
Thus, it is our main task to verify that in the not-conjugate-to-linear setting, the box dimension appearing in Theorem \ref{Theorem AAE and box dim} cannot vanish. To this end, we adopt a variant of Naud's non local integrability condition \cite[Definitions 2.1-2.2]{Naud2005exp}:

\begin{Lemma} \label{Lemma sufficent holds}
 If there exist $(\xi,\zeta, \omega,\eta) \in \left( \mathcal{A}^\mathbb{N} \right)^4$ such that the function 
$$g:I\rightarrow \mathbb{R},\quad g(x) = E(\xi,\zeta,x ,x_\eta )$$
satisfies that $g'(x_\omega)\neq0$, then $\underline{\dim}_B D( \mathcal{A}^\mathbb{N} \times \mathcal{A}^\mathbb{N} \times \mathcal{A}^\mathbb{N} \times \mathcal{A}^\mathbb{N} )>0$.

In particular, $\Phi$ fails the AAE property.
\end{Lemma}
\begin{proof}
The assumption means that  there is some $n\in \mathbb{N}$ such that $g'$ does not vanish on $f_{\omega|_n}(K)$. This means that $g'$ is bi-Lipschitz on $f_{\omega|_n}(K)$.  So,
$$\dim_H  D(\xi,\zeta, \mathcal{A}^\mathbb{N},\eta) = \dim_H  E(\xi,\zeta, K,x_\eta)  \geq  \dim_H E(\xi,\zeta, f_{\omega|_n}(K),x_\eta)$$
$$  = g(f_{\omega|_n}(K)) = \dim_H K >0.$$
The last part of the Lemma now follows via Theorem \ref{Theorem AAE and box dim}.
\end{proof}
We proceed to prove two Claims that together will allow us to verify the conditions of Lemma \ref{Lemma sufficent holds} in our setting. We follow the general strategy of Avila-Gou\"{e}zel-Yoccoz  \cite[Proposition 7.4]{Avila2006yoccoz}, with some modifications due to the possible lack of separation in the IFS.

\begin{Claim} \label{Claim next sufficinet} 
If there exists some $c>0$ such that for infinitely many $n$ there are $\xi=\xi(n),\zeta=\zeta(n) \in \mathcal{A}^\mathbb{N}$ such that for some $x_0=x_0(n)\in K$
$$\left| \frac{d}{dx} \left( \log f_{\xi|_n} '    -  \log f_{\zeta|_n} '   \right) \left(x_0 \right) \right| \geq c$$
then the condition of Lemma \ref{Lemma sufficent holds} holds.
\end{Claim}
\begin{proof}
Suppose the condition of Lemma \ref{Lemma sufficent holds} fails. Then  for every $(\xi,\zeta,\eta) \in (\mathcal{A}^\mathbb{N})^3$ the corresponding function $g$ as in Lemma \ref{Lemma sufficent holds} satisfies $g'(x)=0$ for every $x\in K$ . So,  for all $x\in K$ and every $n$
\begin{eqnarray*}
0&=& g'(x) \\
&=& \lim_k \frac{d}{dx} \left( \log f_{\xi|_k} '   (x) -  \log f_{\zeta|_k} '   (x)  \right) \\
&=& \sum_{k=1} \frac{  f_{\xi_k} '' \circ f_{\xi|_{k-1}} (x) \cdot f_{\xi|_{k-1}} ' (x)   }{f_{\xi_k} ' \circ f_{\xi|_{k-1}} (x) } - \sum_{k=1} \frac{  f_{\zeta_k} '' \circ f_{\zeta|_{k-1}} (x) \cdot f_{\zeta|_{k-1}} ' (x)   }{f_{\zeta_k} ' \circ f_{\zeta|_{k-1}} (x) } \\
&=& \frac{d}{dx} \left( \log f_{\xi|_n} '   (x) -  \log f_{\zeta|_n} '   (x)  \right) \\
&+&  \sum_{k\geq n} \frac{  f_{\xi_k} '' \circ f_{\xi|_{k-1}} (x) \cdot f_{\xi|_{k-1}} ' (x)   }{f_{\xi_k} ' \circ f_{\xi|_{k-1}} (x) } - \sum_{k\geq n} \frac{  f_{\zeta_k} '' \circ f_{\zeta|_{k-1}} (x) \cdot f_{\zeta|_{k-1}} ' (x)   }{f_{\zeta_k} ' \circ f_{\zeta|_{k-1}} (x) }. 
\end{eqnarray*}
Since $\sup_{x\in I, f\in \Phi} \left| \frac{f''(x)}{f'(x)} \right| =O(1)$, we obtain
$$ \left\Vert \frac{d}{dx} \left( \log f_{\xi|_n} ' -  \log f_{\zeta|_n} '  \right) \right\Vert_{\infty,K} = O( \sup_{f\in \Phi} ||f'||_\infty ^{n-1} ).$$
This contradicts our assumptions.
\end{proof}
Here is the final ingredient in our proof:
\begin{Claim} \label{Claim final ing}
If the condition in Claim \ref{Claim next sufficinet} fails then $\Phi$ is $C^r$ conjugate to a  an IFS $\Psi$ such that
$$g''(x) = 0 \text{ for every } x\in K_\Psi \text{ and } g\in \Psi.$$
\end{Claim}
\begin{proof}
Suppose the condition in Claim \ref{Claim next sufficinet} fails. Then for any $\xi,\zeta \in \mathcal{A}^\mathbb{N}$ and any $x\in K$ we have
\begin{equation} \label{Eq AGY}
\lim_n \frac{d}{dx} \log f_{\xi|_n} ' (x) = \lim_n \frac{d}{dx} \log f_{\zeta|_n} ' (x).
\end{equation}
Now, fix $i\in \mathcal{A}$, and let $\bar{i}\in \mathcal{A}^\mathbb{N}$ be the corresponding $\sigma$-periodic point. Fix $x_0 \in I$.  Define a function $\varphi_i:I\rightarrow \mathbb{R}$ via
$$\varphi_i (x):= \lim_n \log f_{\bar{i}|_n} ' (x) - \log f_{\bar{i}|_n} ' (x_0)$$
It is standard  that $\varphi_i$ is $C^{r-1}$. Now, for every $x\in I$ we have
$$\varphi_i (f_i (x)) = \lim_n \log f_{\bar{i}|_n} ' \circ f_i (x) - \log f_{\bar{i}|_n} ' (x_0) = \lim_n \sum_{j\leq n} \log f_i ' \circ f_{\bar{i}|_{j+1}} (x) - \log f_{\bar{i}|_n} ' (x_0) $$
$$= \varphi_i (x) - \log f_i ' (x) + \log f_i ' (x_{\bar{i}}).$$
Therefore, for any $i\in \mathcal{A}$ and any $x\in I$,
\begin{equation} \label{Eq for any i}
\varphi_i \circ f_i (x) = - \log f_i ' (x) + \varphi_i(x)+ \log f_i ' (x_{\bar{i}}).
\end{equation}

Note that we can produce such a function $\varphi_j$ for every $j \in \mathcal{A}$. So, for every $j\in \mathcal{A}$ we define a function $d_j:I\rightarrow \mathbb{R}$ via
$$d_j(x) = \varphi_1(x)-\varphi_j(x).$$
By \eqref{Eq AGY} for every $x\in K$ we have that $\varphi_1' (x) = \varphi_j' (x)$ so $d_j'(x)=0$ for all $x\in K$.  Also, using \eqref{Eq for any i}, for any $x\in I$ and $i\in \mathcal{A}$,
$$ \varphi_1 \circ f_i (x) = \varphi_i \circ f_i (x) +d_i\circ f_i (x) = - \log f_i ' (x)+\varphi_i (x) +d_i \circ f_i (x)+ \log f_i ' (x_{\bar{i}})$$
$$ = - \log f_i ' (x)+\varphi_1 (x)-d_i (x) +d_i \circ f_i (x)+ \log f_i ' (x_{\bar{i}}).$$
To conclude, for every $i\in \mathcal{A}$ the function $F_i:I\rightarrow \mathbb{R}$ defined by
$$F_i (x):= d_i \circ f_i (x) -d_i (x)+ \log f_i ' (x_{\bar{i}})$$
satisfies  that 
\begin{equation} \label{Eq useful}
\varphi_1 \circ f_i (x) = - \log f_i ' (x)+\varphi_1 (x) +F_i (x) \text{ for every } x\in I, \text{ and } F_i ' (x) =0 \text{ for all } x\in K. 
\end{equation}

Finally, let $h:I\rightarrow \mathbb{R}$ be a $C^{r}$ smooth function  that is a primitive of $\exp( \varphi_1 (x))$ on $I$. For every $i\in \mathcal{A}$ define a function $g_i: h(I)\rightarrow h(I)$ via
$$g_i (x) := h \circ f_i \circ h^{-1} : h(I) \rightarrow h(I)$$
and let $\Psi$ be the IFS consisting of the maps $g_i$. Then $\Psi$ is $C^r$ conjugate to $\Phi$.

We claim that $\Psi$ is a linear IFS. Indeed,  by \eqref{Eq useful}, for every $i \in \mathcal{A}$ and every $y\in h(I)$
\begin{eqnarray*}
g_i' (y)  &=& \left( h \circ f_i \circ h^{-1} \right)'(y) \\
&=& \frac{h' \left( f_i \circ h^{-1} (y) \right) \cdot  f_i' (h^{-1}(y)) }{h'(h^{-1}(y))}\\
&=& \exp\left( \varphi_1 \circ f_i \circ h^{-1}(y)  +\log \left( f_i ' \circ h^{-1}(y) \right)  -\varphi_1 \circ h^{-1}(y) \right)  \\
&=& \exp\left( F_i \circ h^{-1} (y) \right).  \\
\end{eqnarray*}
Therefore, for every $y\in h(K)$ we have
$$g_i'' (y) = F_i ' (h^{-1} (y)) \cdot \left( h^{-1} \right)' (y)  \cdot \exp\left( F_i \circ h^{-1} (y) \right)  = 0$$
as $F_i '$ vanishes on $K$ by \eqref{Eq useful}. Since $h(K)$ is the attractor of $\Psi$, the proof is complete.
\end{proof}

\noindent{\textbf{Proof of Theorem \ref{Thm Dolgopyat}}} We  show that in the not-conjugate-to-linear setting  there is some $\alpha>0$ that  satisfies the conditions of Lemma \ref{Lemma 3}. Thus, via the conclusion of Lemma \ref{Lemma 3}, Theorem \ref{Thm Dolgopyat} will follow. Indeed, if this is not the case then  by Lemma \ref{Lemma 4} $\Phi$ has the AAE property. However, since $\Phi$ assumed not to be conjugate to linear, by Claim \ref{Claim final ing} the condition in Claim \ref{Claim next sufficinet} holds true. This in turn implies that the  the condition of Lemma \ref{Lemma sufficent holds} holds true. But by Lemma \ref{Lemma sufficent holds} $\Phi$ cannot have the AAE property. This is a contradiction. The Theorem is proved. $\hfill{\Box}$

\section{An Effective central  limit Theorem for  the derivative cocycle} \label{Section LLT and CLT}
Let $\mathbb{P}=\mathbf{p}^\mathbb{N}$ be a Bernoulli measure on $\mathcal{A}^\mathbb{N}$, and keep the notations and assumptions as in Section \ref{Section transfer}. In this Section we discuss an effective version of the central limit Theorem for a certain random walk driven by the derivative cocycle \eqref{The der cocycle}. This random walk is defined as follows: Denoting by $\sigma: \mathcal{A}^\mathbb{N} \rightarrow \mathcal{A}^\mathbb{N}$ the left shift, for every $n\in \mathbb{N}$ we define a function  on $\mathcal{A}^\mathbb{N}$ via
\begin{equation} \label{Eq for Sn}
S_n(\omega)=-\log f'_{\omega|_n}(x_{\sigma^n(\omega)}).
\end{equation}
Let $X_1: \mathcal{A}^\mathbb{N}\rightarrow \mathbb{R}$ be the random variable
\begin{equation} \label{Eq for X1}
X_1(\omega):= c(\omega_1,\sigma(\omega))=-\log f'_{\omega_1}(x_{\sigma(\omega)}) ,
\end{equation}
and note that our assumptions on $\Phi$ imply that $X_1 \in H^1$. Next, for every integer $n>1$ we define 
$$X_n(\omega) =  - \log f_{\omega_n} ' \left( x_{\sigma^n (x_\omega)} \right)  = X_1 \circ \sigma^{n-1}.$$
Let $\kappa$ be the law of the random variable $X_1$. Then for every $n$, $X_n \sim \kappa$.  By uniform contraction  there exists $D,D'\in \mathbb{R}$ as in \eqref{Eq C and C prime},  
so $\kappa \in \mathcal{P}([D,D'])$. In particular, the support of $\kappa$ is bounded away from $0$. It is easy to see that for every $n \in \mathbb{N}$ and $\omega\in A^\mathbb{N}$ we have
$$S_n(\omega) = \sum_{i=1} ^n X_i (\omega).$$
Thus, in this sense $S_n$ is a random walk.

We proceed to state a version of the central limit Theorem for the random walk $S_n$: For $r> 0$ let $N(0,r ^2)$ be the distribution of a Gaussian random variable with  $0$ mean and variance $r^2$. Also, for any Bernoulli measure $\mathbb{P}$ on $\mathcal{A}^\mathbb{N}$ recall that we write $\chi=\chi_\mathbf{p} = \int c(a,\omega) d\mathbf{p}(a) d\mathbb{P} (\omega)$. 
The  Berry-Esseen type central limit Theorem we now state  follows from a standard application of the Nagaev-Guivarc'h method as presented in the work of Gou\"{e}zel \cite{Gouzel2015limit}:
\begin{theorem} \cite[Theorem 3.7]{Gouzel2015limit} \label{Theorem CLT kasun} Suppose
$\Phi$ satisfies the conditions of Theorem \ref{Main Theorem} and is either Diophantine or not-conjugate-to-linear. Let $\mathbb{P}=\mathbf{p}^\mathbb{N}$ be a Bernoulli measure on $\mathcal{A}^\mathbb{N}$. Then there exists some $r_0=r_0(\mathbf{p})>0$ such that
$$\sup_{z} \left| \mathbb{P}\left( \frac{S_n-n\chi}{\sqrt{n}} \leq z \right) - \left( N(0, r_0 ^2) \leq z \right) \right| = O\left( \frac{1}{n^{\frac{1}{2}}}\right)$$
where $\left( N(0, r_0 ^2) \leq z \right)$ stands for the probability that $N(0,r_0^2)$ is less than $z$.
\end{theorem}
To explain how the setup of Theorem \ref{Theorem CLT kasun} fits into  the conditions of  \cite[Theorem 3.7]{Gouzel2015limit}, we note that since $c$ is a cocycle,  for every $\theta$ the constant function  $\mathbf{1} \in H^1$ satisfies
$$\mathbb{E}\left( e^{2 \pi i \theta S_n} \right) = \mathbb{E} \left( P_{i \theta} ^n \left( \mathbf{1} \right) \right).$$
This confirms  the coding assumption in \cite[Theorem 2.4]{Gouzel2015limit}. The other assumptions of \cite[Theorem 2.4]{Gouzel2015limit} and \cite[Theorem 3.7]{Gouzel2015limit}   follow directly from Theorem \ref{Theorem B-Q 1} and Proposition \ref{Prop compact and small guys} part (2) (where it is explained why here $r_0 >0$).

Finally, we remark that  the very recent works of Fernando-Liverani \cite{Kasun2021Liv} and Cuny-Dedecker-Merlev\`ede \cite{Cuny2021Ded} are closely related to this.   We refer the reader to  \cite{Kasun2021Liv, Kasun020Pene} for an exhaustive bibliography of some further related results.

\section{An effective local limit Theorem with moderate deviations} \label{Section LLT}
Let $\mathbb{P}=\mathbf{p}^\mathbb{N}$ be a Bernoulli measure on $\mathcal{A}^\mathbb{N}$, and keep the notations and assumptions as in Section \ref{Section transfer}. For every $n\in \mathbb{N}$ and $\omega \in \mathcal{A}^\mathbb{N}$  consider the distribution of the centred $n$-step random walk driven by $c$ that starts from $\omega$. This distribution is given by a measure $\mu_{n,\omega}$ on $\mathbb{R}$ such that, for $X\subseteq \mathbb{R}$
$$\mu_{n,\omega} (X) = \int 1_X ( c(a,\omega) - n \chi) d\mathbf{p}^n (a)$$
where, as in Section \ref{Section LLT and CLT}, $\chi$ is the Lyapunov exponent.  Let $G_n$ be the density of the $n$-fold convolution of the Gaussian $N^{*n}(0,r_0 ^2)$ with $r_0$ as in Theorem \ref{Theorem CLT kasun}. That is,
$$G_n(v) = \frac{e^{- \frac{v^2 \cdot r_0 ^2}{2n} }}{\sqrt{2 \pi n}}, \text{ for } v\in \mathbb{R}.$$
The following local limit Theorem is one of the main keys behind the proof of Theorem \ref{Main Theorem}. It is an effective version of a local limit Theorem with moderate deviations due to Benoist-Quint \cite[Theorem 16.1]{Benoist2016Quint}. Recall that $\lambda$ denotes the Lebesgue measure on $\mathbb{R}$.
\begin{theorem} \label{Theorem LLT}
Suppose $\Phi$ satisfies the conditions of Theorem \ref{Main Theorem} and is either Diophantine or not-conjugate-to-linear. Then for every $R>0$ there  is some $\delta=\delta(\mathbf{p},R)>0$ such that for every bounded interval $C\subseteq \mathbb{R}$ 
$$\sup \left\lbrace \left| \frac{\mu_{n,\omega} (C+v_n)}{G_n(v_n)} - \lambda(C) \right|: \, \omega \in \mathcal{A}^\mathbb{N}, \, |v_n|\leq \sqrt{R n \log n} \right\rbrace = O_{\lambda(C)} (\frac{1}{n^\delta}), \text{ as } n\rightarrow \infty.$$
\end{theorem}
Here by $O_{\lambda(C)} (\frac{1}{n^\delta})$ we mean that the multiplicative constant inside the big-$O$ depends on $\lambda(C)$, but we do note that it also depends on other universal multiplicative factors and on $\mathbf{p}$. We do not attempt to give more specific quantitative estimates of the rate, although this is possible.  This result may be extended to other cocycles taking values in vector spaces over $\mathbb{R}$ subject to certain contraction and moment conditions, along with conditions ensuring that the transfer operator contracts fast enough for large frequencies (as in Theorem \ref{Thm Dolgopyat}). Also, similarly to e.g. \cite[Proposition 16.6]{Benoist2016Quint}  Theorem \ref{Theorem LLT} may be adapted to work with a target. However, having the Fourier decay result Theorem \ref{Main Theorem} in our sights, we do not  study these more general situations here.

 The scheme of proof of Theorem \ref{Theorem LLT} is modelled after the proof of Benoist-Quint's local limit Theorem \cite[Theorem 16.1]{Benoist2016Quint}, which is essentially the same as Theorem \ref{Theorem LLT} but without an explicit rate of decay. The proof of Benoist-Quint roughly follows three main steps: First, they prove a version with the interval  $C$ replaced by certain smooth functions on $\mathbb{R}$ \cite[Lemma 16.11]{Benoist2016Quint}. Secondly, they prove that the indicator function $1_C$ admits "good" approximations via such smooth functions \cite[Lemma 16.13]{Benoist2016Quint}. The third and final step is an estimation of $\frac{\mu_{n,\omega} (C+v_n)}{G_n(v_n)}$ for moderately large $v_n\in \mathbb{R}$ using the previous two steps.

 Thus, we will show that the conditions of Theorem \ref{Theorem LLT} yield an effective version of \cite[Lemma 16.11]{Benoist2016Quint}, the local limit Theorem for smooth functions.  Section \ref{Section LLT for smooth}, that contains this result, critically relies on  Theorem \ref{Thm Dolgopyat} to derive certain estimates on an integral that arises from Fourier inversion. This  is inspired by the work of Breuillard \cite[Lemme 3.1]{Breuillard2005llt}, and is related to the analysis of Fernando-Liverani \cite[Theorem 2.4]{Kasun2021Liv}. Then, in Section \ref{Section LLT approx}, we show that the proof of \cite[Lemma 16.13]{Benoist2016Quint} actually yields a polynomial error term. We then combine these into a proof of Theorem \ref{Theorem LLT} in Section \ref{Section proof of LLT}, following along the lines of \cite[Eq. (16.21) and (16.23)]{Benoist2016Quint}.

From this point forward, we  use the standard re-centring  trick as in \eqref{Eq rece cocycle}  and assume $\chi =0$. This will make our computation a bit simpler. Notice that this amounts to changing the cocycle $c$ to a re-centred version $c-n\chi$, which is precisely how the distributions $\mu_{n,\omega}$ 
are defined. From now on, this will be our cocycle.

\subsection{Effective local limit Theorem for smooth functions} \label{Section LLT for smooth}
We proceed to prove a version of our effective local limit Theorem for certain smooth functions. This is in accordance with the strategy of Benoist-Quint \cite[Section 16.2]{Benoist2016Quint}, but via Theorem \ref{Thm Dolgopyat} and ideas going back to Breuillard \cite{Breuillard2005llt} and Stone \cite{Stone1965llt} we make this Theorem  effective.

Fix a non-negative Schwartz function $\alpha$ on $\mathbb{R}$ such that $\lambda(\alpha)=1$, $||\alpha||_\infty \leq 1$, and $\hat{\alpha}$ has compact support, as in  \cite[Definition 16.8 and Remark 16.9]{Benoist2016Quint}. For every $\epsilon>0$ we define
$$\alpha_\epsilon (v) : = \frac{1}{\epsilon} \alpha(\frac{v}{\epsilon}).$$
Fix a bounded interval $C$ and define
$$ \psi_{\epsilon,C}(v):= \int \alpha_\epsilon (w)1_C (v-w)d \lambda(w) = (\alpha_\epsilon \lambda)*1_C$$
which is still a non-negative Schwartz function \cite[Page 268]{Benoist2016Quint}.  For $f\in C^k(\mathbb{R})$ let
$$C^k(f) = \max_{0\leq j \leq k} \parallel f^{(j)} \parallel_{L^1}, \text{ and (even for more general functions) } \hat{f}(\theta) = \int e^{-i\theta x} f(x)dx$$ 
and note that  for every integer $k\geq 1$ 
\begin{equation} \label{Bound for Fourier}
||\widehat{\psi_{\epsilon,C}} ||_\infty \leq \lambda(C), \quad ||\psi_{\epsilon,C}||_\infty \leq \frac{\lambda(C)}{\epsilon}, \quad C^k ( \psi_{\epsilon,C}  ) \leq \lambda(C)\cdot  C^k ( \alpha_{\epsilon}  ) \leq O_\alpha \left( \frac{\lambda(C)}{\epsilon^{k}} \right).
\end{equation}
Recall the notations $\mu_{n,\omega}, G_n$ introduced before Theorem \ref{Theorem LLT}. The following is an effective version of \cite[Lemma 16.11]{Benoist2016Quint}:
\begin{theorem}  \label{Theorem smooth}
Let $\Phi$ be as in Theorem \ref{Theorem LLT} and let $\mathbb{P} = \mathbf{p}^\mathbb{N}$ be a Bernoulli measure. Let $\ell = \alpha+1$, where $\alpha$ is as in Theorem \ref{Thm Dolgopyat}. 

Then for every $r\geq 2$ there exists $\delta = \delta(r)>0$ such that, setting $k= \lceil\ell\cdot r+2\rceil$, for every $\epsilon>0$ we have
$$\sup_{\omega \in \mathcal{A}^\mathbb{N}} \left| \mu_{n,\omega} (\psi_{\epsilon,C}) - \lambda( \psi_{\epsilon,C} \cdot G_n) \right| \leq O_{\lambda(C)} \left(\frac{1}{n^\delta} \right) \cdot \lambda( \psi_{\epsilon,C} \cdot G_n) + \frac{1}{\epsilon^k} \cdot  O_{\lambda(C)} \left(\frac{1}{n^{\frac{r}{2}}} \right)$$

\end{theorem}
Notice that in Theorem \ref{Theorem smooth} the dependence on  $\epsilon$ is explicit in the second error term - this will be important later on. Also, here  the sequence $\epsilon_n$ as in \cite[Lemma 16.11]{Benoist2016Quint} is the polynomially decaying sequence  $O_{\lambda(C)} \left(\frac{1}{n^{\delta}} \right)$. We can, in fact, indicate a more precise rate - see \eqref{Eq specific rate} below.

To prove Theorem \ref{Theorem smooth} we utilize Theorem \ref{Thm Dolgopyat} and Proposition \ref{Prop compact and small guys} to establish Theorem \ref{Lemma key 2}, a Breuillard \cite[Lemme 3.1]{Breuillard2005llt} type estimate on large frequencies of an integral that arises via Fourier inversion on $\mu_{n,\omega} (\psi_{\epsilon,C}) $. In Section \ref{Section proof of smooth} we show how to derive Theorem \ref{Theorem smooth} from this estimate.

\subsubsection{A Breuillard type estimate}
We begin by deriving the following Corollary from Theorem \ref{Thm Dolgopyat}:
\begin{Corollary} \label{Coro Dima}
Let $\alpha,\beta,C>0$  be as in Theorem \ref{Thm Dolgopyat}.  If $|\theta|>1$ and $n \in \mathbb{N}$ satisfy
$$n >  \log |\theta| \cdot \beta \cdot 2.$$
Then 
$$ ||P_{i\theta} ^n (\mathbf{1}) ||_\infty \leq   e^{- \frac{n\cdot C}{|\theta|^{\alpha+1}}}.$$
\end{Corollary}
\begin{proof}
Let $n_0= [\beta\cdot \log |\theta|]$. First, for every $k\in \mathbb{N}$ and $|\theta| > 1$ we have by Theorem \ref{Thm Dolgopyat}
\begin{equation*} 
|| P_{i\theta} ^{n_0 \cdot k} (\mathbf{1}) ||_{(\theta)} \leq ||P_{i\theta} ^{n_0 } ||_{(\theta)} ^k \cdot ||\mathbf{1} ||_{(\theta)} \leq \left( 1- \frac{C}{|\theta|^\alpha} \right)^k \cdot 1 \leq e^{-\frac{Ck}{|\theta|^\alpha}}.
\end{equation*}
Note that in \cite[Section 6]{Dolgopyat1998rapid} the choice of $C_6>0$ as in \eqref{Eq norm dol theta} is made so that for every $r\in \mathbb{N}$,
$$||P_{i \theta} ^r ||_{(\theta)} \leq 1.$$
Now, write  $n = k\cdot n_0+r$ where  $k = \left[ \frac{n}{n_0} \right]$. Via the last two displayed equations we see that
$$||P_{i\theta} ^n (\mathbf{1}) ||_\infty  \leq  ||P_{i\theta} ^n (\mathbf{1}) ||_{(\theta)}  \leq  ||P_{i\theta} ^{n_0 \cdot k} (\mathbf{1}) ||_{(\theta)}   \leq   e^{ \frac{-k}{|\theta|^{ \alpha}}} \leq  e^{ \frac{-n}{|\theta|^{ \alpha+1}}}$$
where in the last inequality we use that $|\theta|>1$.

\end{proof}

Our analysis now allows to estimate one crucial quantity that will come up in the proof of Theorem \ref{Theorem smooth}.  We retain the assumption that the cocycle $c$ is already re-centred, so that \eqref{Eq rece cocycle} holds. The following Theorem is inspired by the work of Breuillard \cite[Lemme 3.1]{Breuillard2005llt} - in fact, \textit{it is} essentially \cite[Lemme 3.1]{Breuillard2005llt} put in our setting. 

\begin{theorem} \label{Lemma key 2}
Let $r>0$ and $\ell'=\alpha +1$ where $\alpha$ is as in Theorem \ref{Thm Dolgopyat}. Then there is a constant $D(r,\mathbf{p})>0$ such that for every $D>D(r,\mathbf{p})$ we have
\begin{equation}
\int_{|\theta|\geq \sqrt{\frac{D\log n}{n}}} \hat{f}(\theta)\cdot P_{i\theta} ^n (\mathbf{1})(\omega)d \theta = C^k (f)\cdot  o_\mathbf{p} (\frac{1}{n^r})
\end{equation}
uniformly  in $\omega \in \lbrace1,..,n \rbrace^\mathbb{N}$ and $f\in C^k$ such that $C^k(f)<\infty$ and $k>\ell'\cdot r+1$.
\end{theorem}
The proof of Theorem \ref{Lemma key 2} follows by mimicking  the proof of \cite[Lemme 3.1]{Breuillard2005llt}: First, the role of $\hat{\mu}(x)$ in \cite{Breuillard2005llt} is replaced by $P_{i\theta} (\omega)$ here. Secondly, the estimate of Corollary \ref{Coro Dima} replaces the estimate $|\hat{\mu}(x)|\leq \exp (-C/|x|^l)$ as in \cite{Breuillard2005llt}. Finally, we remark that the estimates of Proposition \ref{Prop compact and small guys} are used to obtain an analogue of \cite[first equation in the proof of Lemme 3.1]{Breuillard2005llt}. With these observations in hand, Theorem \ref{Lemma key 2} follows readily from the proof of \cite[Lemme 3.1]{Breuillard2005llt}.

\subsubsection{Proof of Theorem \ref{Theorem smooth}} \label{Section proof of smooth}
Recall the family of operators $\lbrace N_{ i \theta} \rbrace$ discussed in and before Proposition \ref{Prop compact and small guys}.
We will require the following asymptotic  expansion result from \cite{Benoist2016Quint}.
\begin{Lemma} \cite[Lemma 16.12]{Benoist2016Quint}  \label{Lemma 16.12}
Let $r\geq 2$. There are polynomial functions $A_i$ on $\mathbb{R}$, $0\leq i \leq r-1$ with degree at most $3i$ and no constant term for $i>0$, with values in the space $\mathcal{L}(H^1)$ of bounded endomorphisms of $H^1$ such that:

For any $M>0$, uniformly in $\theta \in \mathbb{R}$ with $|\theta|\leq \sqrt{M \log n}$ and $\varphi \in H^1$ we have $A_0 (\theta)\varphi = N \varphi$ and in $H^1$
$$e^{\frac{\theta^2 r_0 ^2}{2}} \cdot  e^{-i \sqrt{n}\theta\cdot \chi} \cdot  \lambda_{ \frac{i \theta}{\sqrt{n}}} ^n \cdot  N_{\frac{i \theta}{\sqrt{n}}} \varphi = \sum_{i=0} ^{r-1} \frac{A_i (\theta) \varphi}{n^{ \frac{i}{2} }} + O\left( \frac{ (\log n) ^{ \frac{3r}{2}} |\varphi|_1  }{n^{ \frac{r}{2} }}\right).$$
\end{Lemma}

\noindent{ \textbf{Proof of Theorem \ref{Theorem smooth}}} Recall that we are using the re-centred cocycle, so that
$$ \chi= \int c(a,\omega) d\mathbf{p}(a) d\mathbb{P} (\omega) =0.$$
Fix $\omega \in \mathcal{A}^\mathbb{N}$.  By Fourier inversion \cite[Equation (16.13)]{Benoist2016Quint} we have
$$I_n := 2\pi \mu_{n,\omega} (\psi_{\epsilon,C}) = \int  \widehat{\psi_{\epsilon,C}} (\theta) \mathcal{F}_\theta ( \mu_{n,\omega}) (\theta) d\theta = \int  \widehat{\psi_{\epsilon,C}} (\theta) P_{i\theta} ^n (\mathbf{1})(\omega) d\theta.$$
We will decompose $I_n$ as $I_n = I_n ^2 + I_n ^3 + I_n ^4$. Note that unlike \cite[Page 266]{Benoist2016Quint} we have no need for $I_n ^1$, but we keep the notation to make the comparison with \cite{Benoist2016Quint} easier for the reader.

First, let $T>0$ and define
$$I_n ^2: =  \int_{ |\theta|^2 \geq T \frac{\log n}{n}}  \widehat{\psi_{\epsilon,C}} (\theta) P_{i\theta} ^n (\mathbf{1})(\omega) d\theta.$$
If $T\gg D$ as in Theorem \ref{Lemma key 2} then by  \eqref{Bound for Fourier}
$$|I_n ^2|  = C^k (\psi_{\epsilon,C}) \cdot   o( \frac{1}{n^r}) \leq O \left( \frac{\lambda(C)}{\epsilon^{k}} \right) \frac{1}{n^r}. $$

Next, appealing to \cite[Lemma 11.18]{Benoist2016Quint} there is some $\delta \in (0,1)$ such that in a small  neighbourhood of $0$, $P_{i\theta} - \lambda_{i\theta} N_{i\theta}$ has spectral radius $<\delta$. So, via \eqref{Bound for Fourier}, as long as $n$ is large enough,
$$I_n ^3 := \int_{ |\theta|^2 \leq T \frac{\log n}{n}} \widehat{\psi_{\epsilon,C}} (\theta) \left( P_{i\theta} ^n  -\lambda_{i\theta} ^n N_{i\theta} \right) (\mathbf{1})(\omega)  d\theta = O_{\lambda(C)} (\delta^n).$$

It remains to control
$$I_n ^4 := \int_{ |\theta|^2 \leq T \frac{\log n}{n}}  \widehat{\psi_{\epsilon,C}} (\theta) \lambda_{i\theta} ^n N_{i\theta}  (\mathbf{1})(\omega)  d\theta. $$
By Lemma \ref{Lemma 16.12}, since $\chi =0$ we have via \eqref{Bound for Fourier}
$$I_n ^4 = \int_{ |\theta|^2 \leq T \frac{\log n}{n}}  \widehat{\psi_{\epsilon,C}} (\theta) \cdot  \sum_{i=0} ^{r-1} \frac{\widehat{G_n}(\theta)\cdot A_i (\sqrt{n}\theta) \mathbf{1}(\omega) }{n^{ \frac{i}{2} }} d\theta + O_{\lambda(C)} \left( \left( \frac{\log^3 n}{n} \right)^{ \frac{r+1}{2}} \right)$$
where
$$\widehat{G_n}(\theta) = e^{- \frac{n r_0 ^2 \cdot \theta^2}{2}} = \text{ The  Fourier transform of the Gaussian function } G_n.$$
Since for every $0\leq i \leq r-1$ $A_k$ has degree $\leq 3i$ we get
$$ \int_{ |\theta|^2 \geq  T \frac{\log n}{n}}  \widehat{\psi_{\epsilon,C}} (\theta) \cdot   \frac{\widehat{G_n}(\theta)\cdot A_i (\sqrt{n}\theta) \mathbf{1}(\omega) }{n^{ \frac{i}{2} }} d\theta = O_{\lambda(C)} \left( \frac{\log n^{ \frac{3i+1}{2} }}{n^{ \frac{T+i+1}{2}}} \right).$$
So, choosing $T\gg 1$ we find that
\begin{equation} \label{Eq (16.15)}
I_n  = \int_{ \mathbb{R}}  \widehat{\psi_{\epsilon,C}} (\theta) \cdot  \sum_{i=0} ^{r-1} \frac{\widehat{G_n}(\theta)\cdot A_i (\sqrt{n}\theta) \mathbf{1}(\omega) }{n^{ \frac{i}{2} }} d\theta + \frac{1}{\epsilon^k}\cdot  O_{\lambda(C)} \left( \frac{1}{n^{ \frac{r}{2} }} \right).
\end{equation}

Now, for every $0\leq i \leq r-1$ there exists a polynomial function $B_i$ on $\mathbb{R}$ with values in $H^1$ such that $\deg(B_i)\leq 3i$ and for every $\omega \in \mathcal{A}^\mathbb{N}$ the function on $\mathbb{R}$ given by
$$\theta \mapsto e^{-\frac{r_0 ^2 \theta^2}{2}} A_i(\theta)\mathbf{1}(\omega) = e^{-\frac{r_0 ^2 \theta^2}{2}} A_i(\theta)$$ 
is the Fourier transform of the function on $\mathbb{R}$ given by
$$v\mapsto G_1(v)B_i(v)(\omega).$$
So, by \eqref{Eq (16.15)} and Fourier inversion
\begin{equation} \label{Eq (16.16)}
I_n = 2\pi \int_{\mathbb{R}} \psi_{\epsilon,C} (v) \cdot  G_n(v) \sum_{i=0} ^{r-1} \frac{ B_i ( \frac{v}{\sqrt{n}})(\omega) }{n^{ \frac{i}{2} }} \, d \lambda(v)+ \frac{1}{\epsilon^k}\cdot  O_{\lambda(C)}\left( \frac{1}{n^{ \frac{r}{2} }} \right).
\end{equation}

Next, for every $0\leq i \leq r-1$,
$$\int_{|v|^2 \geq T n \log n } \psi_{\epsilon,C} (v) \cdot  G_n(v)  \frac{ B_i ( \frac{v}{\sqrt{n}})(\omega) }{n^{ \frac{i}{2} }} \, d \lambda(v)= O \left( \frac{ \log n^{ \frac{3i+1}{2} }}{ n^{ \frac{T-i}{2} }} \right) ||\psi_{\epsilon,C}||_\infty  $$
$$ \leq  \frac{1}{\epsilon} \cdot  O \left( \frac{ \log n^{ \frac{3i+1}{2} }}{ n^{ \frac{T-i}{2} }} \right) \leq   \frac{1}{\epsilon^k} \cdot  O \left( \frac{ \log n^{ \frac{3i+1}{2} }}{ n^{ \frac{T-i}{2} }} \right) $$
and since $\psi_{\epsilon,C}$ is non-negative,
$$ \int_{|v|^2 \leq  T n \log n } \psi_{\epsilon,C} (v) \cdot  G_n(v)  \frac{ B_i ( \frac{v}{\sqrt{n}})(\omega) }{n^{ \frac{i}{2} }} \, d \lambda(v)= O \left( \frac{ \log n^{ \frac{3i}{2} }}{ n^{ \frac{i}{2} }} \right) \lambda( \psi_{\epsilon,C} \cdot  G_n).$$

So, choosing $T\gg 1$ the leading term in \eqref{Eq (16.16)} is the one with $i=0$. Since $A_0(\theta)=N$ and $N\varphi= \mathbb{P}( \varphi)$ for all $\varphi$, we get $B_0(v)(\omega)=\mathbb{P}(\mathbf{1}) =1$, so that
\begin{equation} \label{Eq specific rate}
I_n = 2\pi \lambda( \psi_{\epsilon,C} \cdot  G_n)+ \lambda( \psi_{\epsilon,C} \cdot  G_n) O \left( \frac{ \log n^{ \frac{3(r-1)}{2} }}{ n^{ \frac{1}{2} }} \right) + \frac{1}{\epsilon^k}  O_{\lambda(C)} \left( \frac{1}{n^{ \frac{r}{2} }} \right).
\end{equation}
Recalling the definition of $I_n$, the Theorem follows.

\subsection{Approximation of $C$ by a smooth function} \label{Section LLT approx}
The following Lemma allows us to relate the quantities as in Theorem \ref{Theorem smooth} to those appearing in Theorem \ref{Theorem LLT}.  Both the statement and the proof are not much different than \cite[Lemma 16.13]{Benoist2016Quint}; our modest contribution is to notice that the proof given in \cite{Benoist2016Quint} can be made effective. We keep the notation $\psi_{\epsilon,C}$ as in the previous Sections.
\begin{Lemma}  \label{Lemma 16.13}
Let $R\geq 0$ be fixed. Then there is some $\delta''=\delta''(R)>0$ such that
$$\sup \left\lbrace \left| \frac{\lambda( \psi_{\epsilon,C+v} \cdot  G_n)}{G_n(v)} - \lambda(C) \right|: \, \omega \in \mathcal{A}^\mathbb{N}, \, |v|\leq \sqrt{R n \log n}, \, \epsilon\in (0,1) \right\rbrace = O(\frac{1}{n^{\delta''}}).$$
\end{Lemma}

\begin{proof}
Fix $n\geq 1, v\in \mathbb{R}$ with $|v|\leq \sqrt{R n \log n}$ and $\epsilon \in (0,1)$. Let
$$J_n := \frac{\lambda( \psi_{\epsilon,C+v} \cdot G_n)}{G_n(v)} - \lambda(C) = \frac{\lambda( \psi_{\epsilon,C+v} \cdot G_n)}{G_n(v)} - \lambda(C+v).$$
Since $\lambda$ is translation invariant and $ \lambda(\alpha_\epsilon)=1$ we get
$$J_n = \int_{\mathbb{R} \times \mathbb{R}} \alpha_\epsilon (w) 1_{C+v}(w'-w)( \frac{G_n (w')}{G_n(v)}-1) \, d \lambda(w) d \lambda(w').$$ 
We decompose this as a sum $J_n = J_n ^1 + J_n ^2$ with
$$J_n ^1 = \int_{|w| \leq n^{\frac{1}{4}}} \alpha_\epsilon (w) 1_{C+v}(w'-w)( \frac{G_n (w')}{G_n(v)}-1) \, d \lambda(w) d \lambda(w').$$
$$J_n ^2 = \int_{|w| \geq n^{\frac{1}{4}}} \alpha_\epsilon (w) 1_{C+v}(w'-w)( \frac{G_n (w')}{G_n(v)}-1) \, d \lambda(w) d \lambda(w').$$

\noindent{ \textbf{Bounding} $J_n ^1$:} Here, for $w,w'$ such that $w'-v\in C+w$ we have, 
$$w'-v=c+w, \text{ where }  ||w||\leq n^{1/4} \text{, } c\in C, \text{ and } |v|\leq \sqrt{R n \log n}$$
so, 
$$|w'-v|= O(n^{\frac{1}{4}}),\quad  \text{ and } |v+w'|= O \left( \left( n \log n \right )^{1/2} \right).$$
Therefore, 
$$| (v+w')(v-w')|=o(n^{\frac{5}{6}}).$$
Finally, 
$$| \frac{G_n (w')}{G_n (v) }-1|= |e^{\frac{ r_0 ^2 \cdot (v+w')(v-w')}{2n}}-1|=O( n^{-\frac{1}{6}}) $$
and so $J_n ^1 = O(n^{-\frac{1}{6}})$. Here, we use that $\int \alpha_\epsilon(x)\,dx=1$. Notice: the norm of $\alpha_\epsilon$ does not affect this term.

\noindent{ \textbf{Bounding} $J_n ^2$:} First, since $|v|\leq \sqrt{R n \log n}$ we obtain
$$ \frac{G_n (w')}{G_n (v) } \leq e^{ \frac{r_0 ^2 v^2}{2n}} \leq n^{ \frac{R}{2}}$$
Next, since $\alpha$ is a Schwartz function there is some $C_9=C_9(\alpha)$ such that
$$\sup_{x\in \mathbb{R}} |\alpha (x)|\leq \frac{C_9}{(1+|x|)^{4R+1}}.$$
So,  
$$J_n ^2 \leq n^{ \frac{R}{2}} \lambda(C+v) \int_{|w|\geq n^{ \frac{1}{4} }} \alpha_\epsilon(w) d\lambda(w) \leq  n^{ \frac{R}{2}} \lambda(C)   \int_{|x|\geq n^{\frac{1}{4}}} \alpha (x) \,dx$$
$$ \leq O_{\lambda(C)} \left(n^{\frac{R}{2}} \right)\cdot \frac{C_9}{n^{(4R-1+1)/4}} = O_{\lambda(C)} \left(\frac{1}{n^{R/2}}\right) $$
this decays polynomially.

Combining the bounds for $J_n ^1$ and $J_n ^2$, we are done.
\end{proof}
Notice that the implicit constant in the $O_{\lambda(C)} (\cdot )$ above also depends on $\alpha$. However, in practice we will always use the same $\alpha$, so this is indeed a universal bound.

\subsection{Proof of Theorem \ref{Theorem LLT}} \label{Section proof of LLT}
Fix $R>0$ and let $r> 3+R$ and let $k=k(r,\ell)$ be as in Theorem \ref{Theorem smooth}. From now on we fix a  Schwartz function $\alpha$ that satisfies
$$ \int_{|w|\geq n^{ \frac{1}{2k}}} \alpha(w) d\lambda(w) \leq n^{- \frac{1}{k} }.$$
For example, this  holds for $\alpha(x)= e^{-x^2 /2}$. We proceed to combine our previous work to obtain Theorem \ref{Theorem LLT}.

\begin{Lemma} (Upper bound) \label{Lemma upper bound}
There is some $\delta_0>0$ such that
$$\sup \left\lbrace \frac{\mu_{n,\omega} (C+v)}{G_n(v)}:\, \omega \in \mathcal{A}^\mathbb{N},\, |v|\leq \sqrt{R n \log n} \right\rbrace  \leq \lambda(C+v) +O(\frac{1}{n^{\delta_0}}).$$
\end{Lemma}
The proof is based on \cite[Proof of Eq. (16.21)]{Benoist2016Quint}, which with our previous analysis is made effective:
\begin{proof} Let $n\in \mathbb{N}$ and notice that for every $w\in \mathbb{R}$ with $|w|\leq n^{- \frac{1}{2k}}$ we have
$$C\subseteq C+B_0 (n^{- \frac{1}{2k}}) +w, \text{ where } B_0 (e) \text{ is the open ball about } 0 \text{ of radius } e>0. $$
Denote $C^{ (n^{- \frac{1}{2k}})}:= C+B_0 (n^{- \frac{1}{2k}})$. So, since we also have for every $\epsilon>0$
\begin{equation} \label{Eq (16.18)}
\mu_{n,\omega} (\psi_{\epsilon,C}) = \int_\mathbb{R} \alpha_\epsilon (w) \mu_{n,\omega} (C+w)d\lambda(w).
\end{equation}
Plugging in $\epsilon = n^{- \frac{1}{k}}$ we obtain
\begin{equation} \label{Eq (16.22)}
(1-n^{- \frac{1}{k}})\cdot \mu_{n,\omega} (C+v) \leq \mu_{n,\omega} \left(\psi_{  n^{- \frac{1}{k}} ,C^{ ( n^{- \frac{1}{2k}})} +v }\right).
\end{equation}
We also recall that
$$G_n(v)^{-1} \leq (2\pi)^{ \frac{1}{2}} n^{ \frac{1+R}{2} }.$$
Applying successively \eqref{Eq (16.22)}, Theorem \ref{Theorem smooth}, Lemma \ref{Lemma 16.13} we see that:
\begin{eqnarray*}
\frac{\mu_{n,\omega} (C+v)}{G_n(v)} & \leq & \frac{1}{1-n^{- \frac{1}{k}}} \cdot  \frac{ \mu_{n,\omega} \left(\psi_{  n^{- \frac{1}{k}} ,C^{  (n^{- \frac{1}{2k}})} +v }\right) }{ G_n (v) } \\
&\leq &   \frac{ \lambda( \psi_{  n^{- \frac{1}{k}} ,C^{  (n^{- \frac{1}{2k}})} +v } \cdot G_n ) + O\left(\frac{1}{n^\delta}\right) \cdot \lambda( \psi_{  n^{- \frac{1}{k}} ,C^{  (n^{(- \frac{1}{2k})}} +v }\cdot G_n) + n^{ \frac{k}{k} } \cdot  O(\frac{1}{n^{\frac{r}{2}}})  }{ G_n (v) }\\
&\times& \frac{1}{1-n^{- \frac{1}{k}}} \\
&\leq & \frac{1}{1-n^{- \frac{1}{k}}} \cdot \left(  \frac{ \lambda( \psi_{  n^{- \frac{1}{k}} ,C^{  (n^{- \frac{1}{2k}})} +v } \cdot G_n )} { G_n (v) } + O\left(\frac{1}{n^{\delta}}\right) + O\left(n^{ 1 + \frac{1+R}{2} } \cdot \frac{1}{n^{\frac{r}{2}}}\right) \right)  \\
& \leq & \frac{1}{1-n^{- \frac{1}{k}}} \cdot \left( \lambda\left( C^{  (n^{- \frac{1}{2k}})} +v \right) + O\left( \frac{1}{n^{ \delta''} }\right) +O\left(\frac{1}{n^{\delta}}\right) + O\left( n^{ \frac{3+R-r}{2}}\right) \right)\\
&\leq & \frac{1}{1-n^{- \frac{1}{k}}} \cdot \left( \lambda\left( C \right)+ O\left( \frac{1}{n^{ \frac{1}{2k} }} \right) + O\left( \frac{1}{n^{ \delta''} }\right) +O\left(\frac{1}{n^{\delta}}\right) + O\left( n^{ \frac{3+R-r}{2}}\right) \right) \\
&=& \left(1+O\left( \frac{1}{n^{ \frac{1}{k}} } \right) \right) \cdot \left( \lambda\left( C \right)+ O\left( \frac{1}{n^{ \frac{1}{2k} }} \right) + O\left( \frac{1}{n^{ \delta''} }\right) +O\left(\frac{1}{n^{\delta}}\right) + O\left( n^{ \frac{3+R-r}{2}}\right) \right) \\
\end{eqnarray*}
By the choice of $r$ we are done.
\end{proof}

The lower bound is an effective analogue of \cite[Proof of Eq. (16.23)]{Benoist2016Quint}:
\begin{Lemma} (Lower bound) \label{Lemma lower bound}
There is some $\delta_1>0$ such that
$$ \inf \left\lbrace \frac{\mu_{n,\omega} (C+v)}{G_n(v)}:\, \omega \in \mathcal{A}^\mathbb{N},\, |v|\leq \sqrt{R n \log n} \right\rbrace \geq \lambda(C+v) -O(\frac{1}{n^{\delta_1}}).$$
\end{Lemma}
\begin{proof}
 Let $n\in \mathbb{N}$ and notice that for every $w\in \mathbb{R}$ with $|w|\leq n^{- \frac{1}{2k}}$ we have
$$\bigcap_{u\in B_0 (n^{- \frac{1}{2k}})} \left( C-u \right)+w \subseteq C.  $$
Let $C_{ (n^{- \frac{1}{2k}}) }:= \bigcap_{u\in B_0 (n^{- \frac{1}{2k}})} \left( C-u \right)$.   Plugging  $\epsilon = n^{- \frac{1}{k}}$ into \eqref{Eq (16.18)} we have
\begin{eqnarray} \label{Eq (16.24)}
\mu_{n,\omega} (C+v)  & \geq & \int_{|w|\leq n^{- \frac{1}{2k}}} \alpha_{n^{- \frac{1}{k}}} (w) \mu_{n,\omega} (C_{ (n^{- \frac{1}{2k}}) } +v+w)d\lambda(w) \\
&\geq & \mu_{n,\omega} \left(\psi_{  n^{- \frac{1}{k}} ,C_{ ( n^{- \frac{1}{2k}} )}+v} \right)- K_n ^1 - K_n ^2 
\end{eqnarray}
where
$$K_n ^1 =  \int_{ n^{- \frac{1}{2k}} \leq |w| \leq n^{ \frac{1}{4}} } \alpha_{n^{- \frac{1}{k}}} (w) \mu_{n,\omega} (C_{ ( n^{- \frac{1}{2k}}) } +v+w)d\lambda(w)$$
$$K_n ^2 =  \int_{  |w| \geq n^{ \frac{1}{4}} } \alpha_{n^{- \frac{1}{k}}} (w) \mu_{n,\omega} (C_{ (n^{- \frac{1}{2k}} )} +v+w)d\lambda(w).$$
First, via the upper bound from Lemma \ref{Lemma upper bound} and the proof of Lemma \ref{Lemma 16.13} we get
\begin{eqnarray*}
\frac{K_n ^1}{G_n (v)} & \leq & \int_{ n^{- \frac{1}{2k}} \leq |w| \leq n^{ \frac{1}{4}} } \alpha_{n^{- \frac{1}{k}}} (w) \frac{G_n (v+w)}{G_n(v)} \left( \lambda (C +v+w) + O\left(\frac{1}{n^{\delta_0}}\right) \right) d\lambda(w)\\
&\leq & \int_{ n^{- \frac{1}{2k}} \leq |w| \leq n^{ \frac{1}{4}} } \alpha_{n^{- \frac{1}{k}}} (w) \left(1+O\left(\frac{1}{n^{\frac{1}{6} }}\right) \right) \cdot  \left( \lambda (C +v+w) + O\left(\frac{1}{n^{\delta_0}}\right) \right) d\lambda(w)\\
&\leq & \left(1+O\left(\frac{1}{n^{ \frac{1}{6} }}\right) \right) \cdot  \left( \lambda (C) + O\left(\frac{1}{n^{\delta_0}}\right) \right) \cdot n^{- \frac{1}{k}}.
\end{eqnarray*}
Secondly, since $|v|\leq \sqrt{R n \log n}$ then as in the second part of the proof of Lemma \ref{Lemma 16.13}
$$ \frac{K_n ^2}{G_n (v)} \leq n^{ \frac{R}{2} }  \int_{  |w| \geq n^{ \frac{1}{4}} } \alpha_{n^{- \frac{1}{k}}} (w) d\lambda(w) =  n^{ \frac{R}{2} }  \int_{  |w| \geq n^{ \frac{1}{4} + \frac{1}{k}} } \alpha (w) d\lambda(w) = O\left( \frac{1}{n^{ R/2 }} \right). $$
Applying successively \eqref{Eq (16.24)}, Theorem \ref{Theorem smooth}, Lemma \ref{Lemma 16.13}  we get
\begin{eqnarray*}
\frac{\mu_{n,\omega} (C+v)}{G_n(v)} & \geq &  \frac{  \mu_{n,\omega} \left(\psi_{  n^{- \frac{1}{k}} ,C_{  (n^{- \frac{1}{2k}})} +v }\right)- K_n ^1 - K_n ^2  }{ G_n (v) } \\
& \geq & \frac{  \lambda( \psi_{  n^{- \frac{1}{k}} ,C_{  (n^{- \frac{1}{2k}})} +v } \cdot G_n ) - O\left(\frac{1}{n^\delta}\right) \cdot \lambda( \psi_{  n^{- \frac{1}{k}} ,C_{  (n^{- \frac{1}{2k}})} +v }\cdot G_n) - n^{ \frac{k}{k} } \cdot  O\left(\frac{1}{n^{\frac{r}{2}}}\right) }{ G_n (v) } \\
&-&   O\left( \frac{1}{n^{ R/2 }} \right)  - O(n^{- \frac{1}{k}}) \\
&\geq& \left(  \frac{ \lambda( \psi_{  n^{- \frac{1}{k}} ,C_{  (n^{- \frac{1}{2k}})} +v } \cdot G_n )} { G_n (v) } - O\left(\frac{1}{n^{\delta}}\right) - O\left(n^{\frac{3+R}{2} } \cdot \frac{1}{n^{\frac{r}{2}}}\right) \right) \\
&-&  O\left( \frac{1}{n^{ R/2 }} \right)  - O(n^{- \frac{1}{k}}) \\
&\geq& \lambda(C) + O\left( \frac{1}{n^{ \delta''} }\right) +O\left(\frac{1}{n^{\delta}}\right) + O\left( n^{ \frac{3+R-r}{2}}\right) -O\left( \frac{1}{n^{ R/2 }} \right)  - O(n^{- \frac{1}{k}}). \\
\end{eqnarray*}
By the choice of $r$ we are done.
\end{proof}
Via Lemma \ref{Lemma lower bound} and Lemma \ref{Lemma upper bound} the proof of Theorem \ref{Theorem LLT} is complete.

\section{An effective conditional local limit Theorem for smooth functions} \label{Section Thm equid}

Let $\Phi$ be an IFS as in Theorem \ref{Main Theorem} that is  either Diophantine or not-conjugate-to-linear. Let $\mathbb{P} = \mathbf{p}^\mathbb{N}$ be a Bernoulli measure on $\mathcal{A}^\mathbb{N}$. In this Section we prove Theorem \ref{Theorem equid}, a  conditional local limit Theorem  which will be the key behind the proof of Theorem \ref{Main Theorem}. This is an effective version of \cite[Theorem 3.7]{algom2020decay}, and  is proved via the effective local limit Theorem \ref{Theorem LLT} and the effective central limit Theorem \ref{Theorem CLT kasun} that we previously discussed.

We first define the following function  on $\mathcal{A}^\mathbb{N}$. Though it resembles one,  it is \textit{not} a stopping time: Recalling \eqref{Eq for Sn}, we let
$$\ \tau_k(\omega):=\min\{n: S_n(\omega)\geq k\chi\}.$$
Note that we allow $k$  to take positive non-integer values. We also recall that $\chi$ is the corresponding Lyapunov exponent. 

Recalling \eqref{Eq C and C prime}, it is clear that for every $k>0$ and $\omega\in A^\mathbb{N}$ we have
$$-\log |f_{\omega|_{\tau_k(\omega)}}'(x_{\sigma^{\tau_k(\omega)}(\omega)})|= S_{\tau_k(\omega)} (\omega) \in [k\chi, k\chi+D'].$$

Next, we introduce some partitions of the space $\mathcal{A}^\mathbb{N}$, that are modelled after  \cite[Definition 3.3]{algom2020decay}:

\begin{Definition} \label{Def R.V. D} 
 Given a finite word $\eta'=(\eta_1 ' ,\cdots, \eta_\ell ') \in \cA^\ell$:
\begin{enumerate}
\item Denote by $A_{\eta'}\subseteq \cA^{\mathbb N}$ the set of infinite words that begin with $\eta'$,
$$A_{\eta'}:= \lbrace \omega: \quad (\omega_1,...,\omega_\ell)= \eta' \}.$$

\item   We define the event
$$A_{k,\eta'}:= \{\omega:\quad \sigma^{\tau_k(\omega)-1}(\omega)\in A_{\eta'}\}.$$


\item Given $k,h'\geq 0$ we denote by $\cA_k^{h'}$ the finite partition of $\cA^{\mathbb N}$ according to the map 
$$\iota_k^{h'} (\omega) =\iota^{h'}\left(\sigma^{\tau_k(\omega)-1}(\omega) \right)$$
where
$$\iota^h (\eta) =\eta|_{ \ttau_{h} (\eta)},\text{ where } \ttau_{h} (\eta) = \min\lbrace n : -\log \max_{x\in I} |f_{\eta|_n} ' (x) | \geq h\cdot \chi \rbrace.$$
\end{enumerate} 
\end{Definition}

 Note that every cell of the partition $\cA_k^{h'}$ is of the form $A_{k,\eta'}$. Given $k,h'\geq 0$ and $\omega \in \cA^{\mathbb{N}}$ we write $\cA_k^{h'} (\omega)$ for the unique $\cA_k^{h'}$ cell that contains $\omega$. For $\mathbb P$-a.e. $\omega$, we denote the conditional measure of $\mathbb{P}$ on the corresponding  cell by $\mathbb{P}_{\cA_k^{h'}(\omega)}$.  Recall that $\lambda$ is the Lebesgue measure on $\mathbb{R}$, and that $X_1$ is defined in \eqref{Eq for X1}.

\begin{Definition} \label{Def Gammak}
Let $k\in \mathbb{N}$ and let $\eta' \in \mathcal{A}^*$ be a finite word. Assuming $\mathbb{P}(A_{k,\eta'})>0$, we define a  probability measure $\Gamma_{A_{k,\eta'}} $ on $[k\chi, k\chi+D']$  by
$$ \Gamma_{A_{k,\eta'}} := \frac{\int_{A_{\eta'}}{\lambda|_{[k\chi,\, k\chi+X_1(\omega')]}}\, d \mathbb P(\omega')}{\int_{A_{\eta'}}{ X_1(\omega')} \, d \mathbb P(\omega')}.$$.
\end{Definition}
The following Lemma is straightforward: 

\begin{Lemma} \cite[Lemma 3.5]{algom2020decay} \label{Lemma Abs. contin of gamma}
If $\mathbb{P}(A_{k,\eta'})>0$ then  $\Gamma_{A_{k,\eta'}}  \ll \lambda_{[k \chi,k\chi+D']}$ with a density  that depends only on $\mathbb{P}$, such that its norm is bounded above by $\frac{1}{D}$ independently of $k$ and $\eta'$
\end{Lemma}

  We now state  an effective version of our previous conditional local limit Theorem \cite[Theorem 3.7]{algom2020decay}. Its effectiveness is what will ultimately allow us to obtain the rate of decay for $\mathcal{F}_q (\nu)$, where $\nu$ is the projection of $\mathbb{P}$ to the fractal (the corresponding self-conformal measure). The idea behind it is to describe some local limit like phenomenon for the random variable $S_{\tau_k}$, and this is achieved subordinate to the partitions $\mathcal{A}_{k} ^{h'}$. We also note that the $h'$ part in $\mathcal{A}_{k} ^{h'}$  is  useful for the linearization argument outlined in Section \ref{Section collecting error terms}.

\begin{theorem}  \label{Theorem equid}
If $\Phi$ satisfies the conditions of Theorem \ref{Main Theorem} and is either Diophantine or not-conjugate-to-linear, then there exists some $\delta_0=\delta_0(\mathbf{p})>0$ such that:

For every $k,h'>0$ there exists a subset $\tA_{k}^{h'}\subseteq \mathcal{A}^\mathbb{N}$ such that:  \begin{enumerate} [label=(\roman*)]
\item $\mathbb P(\tA_{k}^{h'})\geq  1-O \left( \frac{1}{k^{ \frac{1}{4} }} \right) $.
\item for all $\xi\in \tA_{k}^{h'}$, $\mathbb{P}(\cA_k^{h'}(\xi))>0$.
\item  for all $\xi\in \tA_{k}^{h'}$ and for any sub-interval $J\subseteq  [k\chi, k\chi+D']$, 
$$
 \mathbb{P}_{\cA_k^{h'}(\xi)}(S_{\tau_k}\in J)= \Gamma_{\cA_k^{h'}(\xi)} (J)  +O \left( \frac{1}{k^{\delta_0}} \right)
$$
\end{enumerate}
\end{theorem}
All big-$O$ terms should be understood to depend on $\mathbf{p}$. There are two main differences between Theorem \ref{Theorem equid} and \cite[Theorem 3.7]{algom2020decay}: The most substantial one is that  the error terms are explicit and  polynomial in $k$. The second one is that in \cite[Theorem 3.7]{algom2020decay} we work inside cylinders to get pointwise normality (which requires more parameters), but here we only care about Fourier decay which allows us to make the statement simpler.

The proof of Theorem \ref{Theorem equid} is similar to that of Theorem \cite[Theorem 3.7]{algom2020decay}. Let us now  explain how the quantitative estimates we previously obtained can be used to make \cite[Proof of Theorem 3.7]{algom2020decay} effective:

For every $k$ we define the interval 
\begin{equation} \label{Eq Def of Ik}
I_k = [k- \sqrt{k\log k}, k+\sqrt{k\log k}].
\end{equation}
The proof of Theorem \ref{Theorem equid} relies on a decomposition of the left hand side in (iii) as 
\begin{equation}\label{Eq Theorem equid}
 \mathbb{P}_{\cA_k^{h'}(\xi)} \left(  S_{\tau_k} \in J  \right)= \sum_{m \not \in I_k}  \mathbb{P}_{\cA_k^{h'}(\xi)} \left( S_{\tau_k} \in J, \,  \tau_k = m+1 \right)+\sum_{m \in I_k} \mathbb{P}_{\cA_k^{h'}(\xi)} \left( S_{\tau_k} \in J, \,  \tau_k = m+1 \right)\end{equation}

Both terms are respectively treated by  Proposition \ref{Prop equid central} and Proposition \ref{Prop equid local} below, and the Theorem follows. First, we have:
\begin{Proposition} \label{Prop equid central}
There exists a set $\widetilde A$ such that claims (i) and (ii) hold and for all $\xi\in\widetilde A$, 
$$\mathbb{P}_{\cA_k^{h'}(\xi)}( \tau_k -1 \notin I_k ) = O \left( \frac{1}{k^{\frac{1}{4}}} \right).$$
\end{Proposition}
Notice that we are using the abbreviated notation $\widetilde A$ instead of $\tA_{k}^{h'}$. Proposition \ref{Prop equid local} is an effective version of \cite[Proposition 3.12]{algom2020decay}. The key to the proof is showing that
\begin{equation}\label{EqTotalCLT}\mathbb{P}(\tau_k -1 \notin I_k) =  O \left( \frac{1}{\sqrt{k}}\right).\end{equation}
Indeed, once \eqref{EqTotalCLT} is established, the result follows by an application of Markov's inequality, similarly to the end of the proof of  \cite[Proposition 3.12]{algom2020decay}. The proof of \eqref{EqTotalCLT} is rather straightforward and is essentially the same as the proof of \cite[Eq. (15)]{algom2020decay}. The latter proof can now be made effective by replacing the use of the non-effective central limit Theorem \cite[Theorem 12.1]{Benoist2016Quint} with the effective Theorem \ref{Theorem CLT kasun}.

The second term in \eqref{Eq Theorem equid} is treated in the following Proposition:

\begin{Proposition}\label{Prop equid local}
Under the assumptions of Theorem \ref{Theorem equid}, there is some $\delta_0>0$ such that for all $\xi$ in the set $\widetilde A$ from Proposition \ref{Prop equid central},
  $$\sum_{m \in I_k} \mathbb{P}_{\cA_k^{h'}(\xi)} \left(  S_{\tau_k} \in J,\, \tau_k = m+1 \right)=\Gamma_{\cA_k^{h'}(\xi)}(J)  +O\left( \frac{1}{k^{\delta_0}} \right).$$
  \end{Proposition}
This is an effective version of \cite[Proposition 3.13]{algom2020decay}.  Let us  explain how, via Theorem \ref{Theorem LLT},  \cite[Proof of Proposition 3.13]{algom2020decay} can be made effective.  Let $\eta'$ be the finite word such that ${\cA_k^{h'}(\xi)}=A_{k,\eta'}$. Write
\begin{equation*} 
\sum_{m \in I_k} \mathbb{P}_{\cA_k^{h'}(\xi)} \left( S_{\tau_k} \in J, \, \tau_k = m+1 \right)  = \frac{\sum_{m \in I_k}\mathbb{P} \left( S_{\tau_k} \in J, \, \tau_k = m+1,\, \omega \in A_{k,\eta'} \right) }{ \mathbb{P}(A_{k,\eta'})}.
\end{equation*}

As shown in \cite[Proof of Proposition 3.13]{algom2020decay},  each summand in the numerator can be written as 
$$\mathbb{P} \left( S_{\tau_k} \in J, \, \tau_k = m+1,\, \omega \in A_{k,\eta'} \right) = \int_{A_{\eta'}}\mathbb{P}_{\sigma^{-m}(\{\omega'\})}\big( S_m \in J^{\omega'} \big) \, d\mathbb P (\omega'),$$
where the interval $J^{\omega'}$ is defined by
$$J^{\omega'} := [k\chi - X_1(\omega'),\, k\chi) \cap \left( J - X_1(\omega') \right).$$
Note that $J^{\omega'}$ also depends on $k$, but we suppress this in our notation. Let $a_{\omega',k}$ be the left endpoint of $J^{\omega'}$. We now slightly change our notation:
$$ \text{ We write } G_s (\cdot ) \text{ for the density of the Gaussian random variable } N(0,s^2).$$ 
By Theorem \ref{Theorem LLT}, we have for every $\omega'$, using that $m\in I_k$ so  $|a_{\omega',k} - m\chi| \leq 8 \chi \sqrt{m \log m}$,
\begin{equation*}
\mathbb{P}_{\sigma^{-m}(\{\omega'\})}\big( S_m \in J^{\omega'})=G_{\sqrt m r}(a_{\omega',k}-m\chi)\cdot    \left( \lambda(J^{\omega'}) +O \left( \frac{1}{m^\delta} \right) \right).
\end{equation*} 
Using that $m\in I_k$ and  applying \cite[Lemma 3.14]{algom2020decay}, there is some $\delta'$ such that 
\begin{equation} \label{EqLLTComponent}
\mathbb{P}_{\sigma^{-m}(\{\omega'\})}\big( S_m \in J^{\omega'}) = G_{\sqrt k r}((m-k+\beta)\chi) \cdot  \left( \lambda(J^{\omega'}) +   O( \frac{1}{k^{\delta'}}) \right)
\end{equation}
 for all $\omega'\in A_{\eta'}$, $m\in I_k$ and $\beta\in[0,1)$ as $k\to\infty$. 
 
 From here, one simply swaps \cite[Equation (20)]{algom2020decay} for \eqref{EqLLTComponent}, and essentially the same proof given for \cite[Proposition 3.13]{algom2020decay}) yields Proposition \ref{Prop equid local}.

\section{Proof of Theorem \ref{Main Theorem}} \label{Section proof}
In this Section we prove Theorem \ref{Main Theorem}. Fix an IFS $\Phi$ as in Theorem \ref{Main Theorem} that is either Diophantine or not-conjugate-to-linear, and let $\mathbb{P} = \mathbf{p}^\mathbb{N}$ be a Bernoulli measure on $\mathcal{A}^\mathbb{N}$.  As in \cite{algom2020decay}, we first require a preliminary step -  an adaptation of Theorem \ref{Theorem equid} to Fourier modes. This is the content of the following Section.
\subsection{Application of Theorem \ref{Theorem equid} to Fourier modes}
Fix a Borel probability measure $\rho\in \mathcal{P}(\mathbb{R})$.  For every $q\in \mathbb{R}$ we define a function $g_{q,\rho}:\mathbb{R}\rightarrow \mathbb{R}$ via
$$g_{q,\rho} (t) = \left| \mathcal{F}_q \left( M_{e^{-t}} \rho \right) \right|^2, \text{ where for any } s,x\in \mathbb{R}, \,   M_s(x):=s\cdot x.$$
The following is a  version of Theorem \ref{Theorem equid} for Fourier modes instead of intervals:
\begin{theorem} \label{Theorem equid 2}
Fix parameters $q,k,h'$ with $k,h',|q|>0$, and let $\rho\in \mathcal{P}(\mathbb{R})$ be a measure such that 
$$\diam \left( \supp \left( \rho \right) \right) = O(e^{-h'\chi}).$$
Then for every $\xi \in \tA_{k}^{h'} \subseteq \mathcal{A}^\mathbb{N}$, for $\delta_0>0$ as in Theorem \ref{Theorem equid},   we have
\begin{equation*}
\left| \mathbb{E}_{\cA_k^{h'}(\xi)} \left[ g_{q,\rho} (S_{\tau_k(\omega)}) \right] - \int_{k\chi} ^{k\chi +D'} g_{q,\rho}(x) d \Gamma_{\cA_k^{h'}(\xi)} (x) \right|\leq O \left( \frac{2}{q e^{-(k+h')\chi}}+ (q e^{-(k+h')\chi})^2 \frac{1}{k^{\delta_0}} \right).
\end{equation*}
\end{theorem}
This  is an analogue of \cite[Theorem 4.1]{algom2020decay}. The main difference is that we swap the $o_k$ term in \cite[Theorem 4.1]{algom2020decay}, which is the non-effective rate at which \cite[Theorem 3.7]{algom2020decay} holds, for a more explicit bound in terms of our parameters and the effective rate at which Theorem \ref{Theorem equid} holds. To sketch the proof, note that the Lipschitz norm of the function
$$t\in [ k\chi, k\chi +D'] \mapsto g_{q,\rho} (t)$$
is $4 \pi q e^{-\chi k}\cdot \text{diam} \left( \supp(\rho) \right)$. This allows for the construction of a $O\left( \frac{1}{q e^{-(k+h')\chi}} \right)$-approximating step-function on this interval in the sup-norm, with $(q e^{-(k+h')\chi})^2 $-steps. Each step corresponds to an indicator function of some sub-interval of $[ k\chi, k\chi +D']$,  where Theorem \ref{Theorem equid} holds with a uniform rate of $ O\left( k ^{-\delta_0} \right)$. This implies the Theorem. For a detailed outline of this sketch, see \cite[Proof of Theorem 4.1]{algom2020decay}.

\subsection{Collecting error terms} \label{Section collecting error terms}
 Let $\nu$ be the self-conformal measure that arises by projecting $\mathbb{P}$ to the fractal $K$. Fix parameters $q,k,h'$ where $q$ will be the frequency of the Fourier transform of $\nu$, and $k,h'$ positive numbers that will depend on $q$. In this Section we will bound $\mathcal{F}_q (\nu)$ via a sum of certain error terms that depend variously on $|q|,k,h'$.   These error terms will arise from three main sources: Linearization, the local limit Theorem, and an oscillatory integral. This is  analogues to  \cite[Section 4.2]{algom2020decay}, so we exclude some of  the proofs (but we will indicate exactly what we are using from  \cite{algom2020decay}).  For brevity,  let $\tilde{A}$ be the set $\tA_{k}^{h'}$ as in Theorem \ref{Theorem equid} for our parameters. The most technically involved estimate arises from a linerization scheme, whose outcome is summarized in the following Theorem. Here, and throughout this Section, all big-$O$ terms should be understood to depend on $\mathbb{P}$ and $\Phi$. For every $s,x\in \mathbb{R}$ we denote the scaling map by $M_s(x):=s\cdot x$, and recall that $\Phi$ is $C^{r}$ smooth, $r\geq 2$.
\begin{theorem} (Linearization) \label{Theorem lin}
There is some integer $P>1$ such that for any $\beta \in (0,1)$,
$$\left| \mathcal{F}_q (\nu) \right|^2 \leq \sum_{|\rho|\leq 2P} \int_{\xi \in \tilde{A}}  \mathbb{E}_{\cA_k^{h'}(\xi)}  \left| \mathcal{F}_q \left( M_{e^{-S_{\tau_k(\omega)} (\omega)}}\circ  f_{\bar{\eta'}} \circ f_\rho  \nu \right) \right|^2 \, d \mathbb{P}( \xi) $$
$$+ O\left( \frac{1}{k^{ \frac{1}{4}} } \right) + O\left(|q|\cdot e^{-(k+h') \chi-\beta\cdot h' \chi}\right),$$
where:
\begin{enumerate}
\item For every $\xi \in \tilde{A}$ the $\eta'$ inside the integral corresponds to  the cell $\cA_k^{h'}(\xi) = A_{k,\eta'}$.

\item For every $\eta'$ we define  $\bar{\eta'} := \eta'|_{|\eta'|-P}$, the prefix of $\eta'$ of length $|\eta'|-P$.

\item  There is a global constant $C'>1$ such that for all $\bar{\eta'}$ and $\rho$ as above,
$$|\left( f_{\bar{\eta'}} \circ f_{\rho} \right) ' (x)| = \Theta_{C'} \left( e^{-h'\chi} \right), \quad \forall x\in I.$$
\end{enumerate}
\end{theorem}

\begin{proof}
This is a combination of  \cite[Lemma 4.3, Lemma 4.4, Claim 4.5, and Corollary 4.6]{algom2020decay}, and since $\Phi$ is assumed to be orientation preserving.  
\end{proof}

Now, fix some $\rho$ with $|\rho|\leq 2P$ and consider the corresponding term in Theorem \ref{Theorem lin},
$$ \int_{\xi \in \tilde{A_\eta}}  \mathbb{E}_{\cA_k^{h'}(\xi)}  \left| \mathcal{F}_q \left( M_{e^{-S_{\tau_k(\omega)} (\omega)}}\circ  f_{\bar{\eta'}} \circ f_\rho  \nu \right) \right|^2 \, d \mathbb{P}( \xi).$$
We now appeal to the local limit  Theorem \ref{Theorem equid 2} for every event ${\cA_k^{h'}(\xi)}$   separately. To do this, we notice that by Theorem \ref{Theorem lin}, for every $f_{\bar{\eta'}} \circ f_\rho$ involved 
$$\diam \left( \supp \left( f_{\bar{\eta'}} \circ f_\rho    \nu \right) \right) =O(e^{-h' \chi}).$$
Notice that the  error term in Theorem \ref{Theorem equid 2} is $O\left(\frac{2}{|q| e^{-(k+h')\chi}}+ (|q| e^{-(k+h')\chi})^2 \frac{1}{k^{\delta_0}}\right)$ independently of the event ${\cA_k^{h'}(\xi)}$. So, 
$$ \int_{\xi \in \tilde{A_\eta}}  \mathbb{E}_{\cA_k^{h'}(\xi)} \left| \mathcal{F}_q \left( M_{e^{-S_{\tau_k(\omega)} (\omega)}}\circ f_{\bar{\eta'}} \circ f_\rho   \nu \right) \right|^2 \, d \mathbb{P}( \xi) $$
$$\leq \int_{\xi \in \tilde{A_\eta}}   \int_{k\chi} ^{k\chi +D'} \left| \mathcal{F}_q \left(M_{e^{-x}} \circ f_{\bar{\eta'}} \circ f_\rho    \nu \right) \right|^2 \, d \Gamma_{\cA_k^{h'}(\xi)} (x) \,  d \mathbb{P}( \xi) +O\left(\frac{2}{|q| e^{-(k+h')\chi}}+ (|q| e^{-(k+h')\chi})^2 \frac{1}{k^{\delta_0}}\right)$$
Since this is true for every $\rho$ with $|\rho|\leq 2P$, combining with Theorem \ref{Theorem lin} we see that
$$\left| \mathcal{F}_q (\nu) \right|^2 \leq \sum_{|\rho|\leq 2P} \int_{\xi \in \tilde{A}}   \int_{k\chi} ^{k\chi +D'} \left| \mathcal{F}_q \left(M_{e^{-x}}\circ f_{\bar{\eta'}} \circ f_\rho   \nu \right) \right|^2 d \Gamma_{\cA_k^{h'}(\xi)} (x)   d \mathbb{P}( \xi)  $$
$$+O\left(\frac{2}{|q| e^{-(k+h')\chi}}+ (|q| e^{-(k+h')\chi})^2 \frac{1}{k^{\delta_0}}\right)+ O \left( \frac{1}{k^{\frac{1}{4}}} \right) + O\left(|q|\cdot e^{-(k+h') \chi-\beta\cdot h' \chi}\right).$$

Recall that by Lemma \ref{Lemma Abs. contin of gamma}, the probability measure $\Gamma_{\cA_k^{h'}(\xi)}$ is absolutely continuous with respect to the Lebesgue measure on $[k \chi, k \chi +D']$, such that  the norm of its density function is  uniformly bounded by $\frac{1}{D}>0$ independently of all parameters. Using this fact,  we obtain
$$\left| \mathcal{F}_q (\nu) \right|^2 \leq \sum_{|\rho|\leq 2P} \int_{\xi \in \tilde{A_\eta}} \left(   \int_{k\chi} ^{k\chi +D'} \left| \mathcal{F}_q \left(M_{e^{-z}} \circ f_{\bar{\eta'}} \circ f_\rho \nu \right) \right|^2 \cdot \frac{1}{D} dz   \right) d \mathbb{P}( \xi) $$
$$+O\left(\frac{2}{|q| e^{-(k+h')\chi}}+ (|q| e^{-(k+h')\chi})^2 \frac{1}{k^{\delta_0}}\right)+ O \left( \frac{1}{k^{\frac{1}{4}}} \right) + O\left(|q|\cdot e^{-(k+h') \chi-\beta\cdot h' \chi}\right).$$
This leads us to the last error term, that comes from the sum of the oscillatory integrals as in the equation above:
\begin{Proposition} (Oscillatory integral) \label{Prop osc}
For every $|\rho|\leq 2P, \xi\in \tilde{A}$, and every $r>0$ we have
$$\int_{k\chi} ^{k\chi +D'} \left| \mathcal{F}_q \left(M_{e^{-z}} \circ f_{\bar{\eta'}} \circ f_\rho \nu \right) \right|^2 \cdot \frac{1}{D} dz \leq O \left( \frac 1  {r |q| e^{-(k+h')\chi}} + \sup_y\nu(B_r(y)) \right).$$ 
\end{Proposition} 
\begin{proof}
This follows from Hochman's Lemma \cite[Lemma 3.2]{Hochman2020Host} about the average of the Fourier transform of scaled measures. Here we are using it in the form \cite[Lemma 2.6]{algom2020decay}, and are also making use of the fact that there is a global constant $C'>1$ such that for all $\bar{\eta'}$ and $\rho$ as above,
$$|\left( f_{\bar{\eta'}} \circ f_{\rho} \right) ' (x)| = \Theta_{C'} \left( e^{-h'\chi} \right), \quad \forall x\in I$$
which follows from Theorem \ref{Theorem lin}. See \cite[Pages 41-42]{algom2020decay} for more details.
\end{proof}

\subsection{Conclusion of the proof}
Following the argument in Section \ref{Section collecting error terms}, we bounded $\left| \mathcal{F}_q (\nu) \right|^2$ by the sum of the following terms. Every term  is bounded with dependence on the Bernoulli measure $\mathbb{P}=\mathbf{p}^\mathbb{N}$ and some fixed $\beta \in (0,1)$. For simplicity, we ignore global multiplicative constants, so we omit the big-$O$ notation. Recall that $\delta_0>0$ is as in Theorem \ref{Theorem equid}:

Linearization - Theorem \ref{Theorem lin}:
$$   |q|e^{-(k+h')\chi}e^{-\beta h'\chi};$$

Local limit Theorem  - The discussion in between Theorem \ref{Theorem lin} and Proposition \ref{Prop osc}:
$$ \frac{2}{|q| e^{-(k+h')\chi}}+ (|q| e^{-(k+h')\chi})^2 \frac{1}{k^{\delta_0}};$$

Oscillatory integral: Via Proposition \ref{Prop osc}, for every $r>0$,
$$\frac 1  {r |q| e^{-(k+h')\chi}} + \sup_y\nu(B_r(y)).$$

\noindent{ {\bf Choice of parameters :}}  For $|q|$  large we choose $k=k(|q|)$ and $h'=\sqrt{k}$ such that
$$|q|=  k ^{\frac{\delta_0}{4}} \cdot e^{(k+h')\chi}. $$ 
Fix $r=k^{\frac{-\delta_0}{8}}$. Then we get:

Linearization:
$$   |q|e^{-(k+h')\chi}e^{-\beta h'\chi}= k^{\frac{\delta_0}{4}}\cdot e^{-\beta\sqrt{k}\chi}, \quad  \text{ This decays exponentially fast in } k. $$

Local limit Theorem:
$$ \frac{2}{|q| e^{-(k+h')\chi}}+ (|q| e^{-(k+h')\chi})^2 \frac{1}{k^{\delta_0}} = \frac{2}{k^{\frac{\delta_0}{4}}} + k^{\frac{\delta_0}{2}} \cdot \frac{1}{k^{\delta_0}}, \quad  \text{ This decays polynomially fast in } k. $$

Oscillatory integral: There is some $d=d(\nu)$ such that
$$\frac 1  {r |q| e^{-(k+h')\chi}} + \sup_y\nu(B_r(y)) \leq  \frac{k^{\frac{\delta_0}{8}}}{k^{\frac{\delta_0}{4}}} +  k^{  \frac{-d \cdot \delta_0}{8}}, \quad  \text{ This decays polynomially fast in } k.$$
Here we made use\footnote{We note that although \cite[Proposition 2.2]{feng2009Lau} is stated for self-similar IFSs, essentially the same proof works in the case of  general $C^r$ smooth IFS's.} of \cite[Proposition 2.2]{feng2009Lau}, where it is shown  that there is some $C>0$ such that for every $r>0$ small enough
$$\sup_y\nu(B_r(y)) \leq Cr^d.$$
Finally, by summing these error terms we see that for some $\alpha=\alpha(\nu)>0$ we have $|\mathcal{F}_q (\nu)|=O(\frac{1}{k^\alpha})$. Since as $|q|\rightarrow \infty$ we have $k\geq O(\log |q|)$ our claim follows.
\bibliography{bib}{}
\bibliographystyle{plain}

\end{document}